\documentclass{amsart}

\usepackage{amsmath}
\usepackage{amssymb}
\usepackage{amsthm}
\usepackage{dsfont}
\usepackage{lineno}
\usepackage{subfigure}
\usepackage{graphicx}
\usepackage{multirow}
\usepackage{color}
\usepackage{xspace}
\usepackage{pdfsync}
\usepackage{hyperref} 
\usepackage{algorithmicx}
\usepackage{algpseudocode}
\usepackage{algorithm}
\usepackage{mathtools}
\usepackage{enumitem}
\graphicspath{{Figures/}}

\DeclarePairedDelimiter\floor{\lfloor}{\rfloor}
\newcommand{\bq}{\begin{equation}}
\newcommand{\eq}{\end{equation}}
\newcommand{\R}{\mathbb{R}}
\newcommand{\abs}[1]{\left\vert#1\right\vert}

\newcommand{\grad}{\nabla}

\newcommand{\bO}{\mathcal{O}}

\newcommand{\one}{\mathds{1}}

\newcommand{\MA}{Monge-Amp\`ere\xspace}
\newcommand{\test}{\mathcal{T}}

\algnewcommand{\LineComment}[1]{\State \(\triangleright\) #1}

\newtheorem{theorem}{Theorem}[section]
\newtheorem{lemma}[theorem]{Lemma}

\theoremstyle{definition}
\newtheorem{definition}[theorem]{Definition}

\theoremstyle{remark}
\newtheorem{remark}[theorem]{Remark}

\numberwithin{equation}{section}

\makeatletter
\newcommand\appendix@section[1]{%
\refstepcounter{section}%
\orig@section*{Appendix \@Alph\c@section: #1}%
}
\let\orig@section\section
\g@addto@macro\appendix{\let\section\appendix@section}
\makeatother

\begin{document}

\title[A viscosity framework for Pogorelov solutions]{A viscosity framework for computing Pogorelov solutions of the Monge-Amp\`ere equation}

\author{Jean-David Benamou}
\address{INRIA, Domaine de Voluceau, BP 153 le Chesnay Cedex, France}
\email{jean-david.benamou@inria.fr}

\author{Brittany D. Froese}
\address{Department of Mathematics and ICES, University of Texas at Austin,  1 University Station C1200, 
Austin, TX 78712}
\email{bfroese@math.utexas.edu}

\thanks{The  authors gratefully acknowledge the support of the ANR through the project ISOTACE (ANR-12-MONU-0013) and INRIA through the ``action exploratoire'' MOKAPLAN
}



\date{\today}

\dedicatory{}

\begin{abstract}
We consider the Monge-Kantorovich optimal transportation problem between two measures, one of which is a weighted sum of Diracs. This problem is traditionally  solved 
using expensive geometric methods.
It can also be reformulated as an elliptic partial differential equation known as the Monge-Amp\`ere equation.  
However, existing numerical methods for this non-linear PDE 
require the measures to have finite density. We introduce a new formulation that couples the viscosity and Aleksandrov solution definitions and show that it is equivalent to the original problem.  Moreover, we describe a local reformulation of the subgradient measure at the Diracs, which makes use of one-sided directional derivatives.  This leads to a consistent, monotone discretisation of the equation.  Computational results demonstrate the correctness of this scheme when methods designed for conventional viscosity solutions fail.
\end{abstract}

\maketitle

\section{Introduction}\label{sec:intro}
The Monge-Kantorovich optimal transportation problem between two measures has continued to be an active area of research since it was first formulated in 1781~\cite{Monge,Kantorovich1,Kantorovich2}.  In this article, we focus on the case where one of the measures consists of a weighted sum of Dirac measures.  This is a problem that arises in applications such as astrophysics~\cite{FrischUniv}, meteorology~\cite{Cullen},  and reflector design~\cite{Merigot,Oliker_beam}. More generally, it has been used to study geometric properties of the Monge-Amp\`ere equation ~\cite{Pogorelov}.  

As explained below,  the  numerical solution of this problem remains a  challenge.  While the optimal transportation problem can be reformulated in terms of an elliptic partial differential equation (PDE) known as the \MA equation~\cite{Brenier}, existing numerical methods do not apply to the type of weak solution that is needed when the data involves Dirac measures.  See~\cite{FengReview} for a recent review of this field.

In this article, we describe a new viscosity formulation of the \MA equation that allows for one of the measures to be a weighted sum of Diracs.  We show that solving this equation is equivalent to solving the original optimal transportation problem in the generalised (Pogorelov) sense.  Furthermore, we provide a consistent, monotone discretisation of the equation.  Computational experiments validate the resulting numerical method.

\subsection{Optimal transportation}

Our motivation stems from 
the ``semi-discrete'' Monge-Kantorovich Optimal Transport Mapping (OTM)  between   $\nu $, a compactly supported measure  on $\Omega\subset\R^d$,  and  an atomic measure  composed of $K$ weighted Dirac masses, 
\begin{equation} 
\label{DM}
\mu  = \sum_{k=1}^{K} \alpha_k \, \delta_{d_k} ,
\end{equation}
where $\{d_k\}$ are $K$ given points in $\R^d$, $\sum_k \alpha_k = \int_{supp(\nu)} d\nu$, and $\alpha_k  > 0 $ .

In order to simplify the presentation in this paper,  we will restrict our attention to the Lebesgue measure on the unit ball 
 \begin{equation} 
 \label{s1}
d\nu(y) = \one_{B(0,1)} \, dy
\end{equation}
and dimension $d=2$. 
In this special case, well-posedness requires that the data satisfy the condition
\[ \sum\limits_k\alpha_k = \abs{B(0,1)} = \pi. \]

\begin{remark}
See Appendix~\ref{app:extension} for the extension to a general absolutely continuous measure $d\nu(y) = g(y) \,dy$ with $g$ a positive, Lipschitz continuous density supported on a convex set.  Additionally, the measure $\mu$ can include a non-zero background density $f(x)$.  
\end{remark}
 
The OTM problem  is to   find a map $T: B(0,1)   \to \R^2$ minimising  the transport cost 
 \begin{equation} 
 \label{c1}
 \int_{B(0,1)}  \| y- T(y)\|^2 d\nu(y) 
 \end{equation} 
 under the constraint that  the map is mass-preserving.  We say that $T$ pushes forward $\nu$ to $\mu$, written as $T\#\nu = \mu$, if
 for every measurable sets $E\in\R^d$,  
 \begin{equation} 
\label{MAW}
 \mu(E) = \nu(T^{-1}(E)).
 \end{equation} 

Formally, the optimal map can be characterised as the gradient of a convex  potential $T = \nabla \phi$ (see \cite{Brenier})  and this potential solves 
the second boundary value problem for the Monge-Amp\`ere equation~\cite{Delanoe,Urbas}. 
 \begin{equation} 
\label{MA}
\begin{cases}
\begin{array}{l}
\det(D^2 \phi(x))  \mu(\nabla \phi(x))  = \nu(x) , \,\,  x \in \Omega  \\
  \nabla \phi(\Omega) \subset \text{supp}(\mu). 
 \end{array}
\end{cases}
 \end{equation} 
Here the constraint on the image of the gradient replaces a more traditional boundary condition.
Under mild regularity assumptions, Caffarelli proved that if the measures $\mu,\nu$ have densities that are bounded away from zero with uniformly convex support, one recovers classical solutions of the PDE~\cite{Caf}.

 In the form (\ref{MA}), the equation does not make sense for 
our  particular choice of atomic  $\mu$~\eqref{DM}.  
We instead resort to  a  a weaker notion of Monge-Amp\`ere ,
which goes back to Pogorelov~\cite{Pogorelov} and coincides with Brenier solutions. 
Since the OTM  only transports mass  onto  the Dirac locations $\{d_k\}$, 
the gradient map will almost everywhere exist with  $\nabla u(y) \in \bigcup\limits_k\{d_k\}$. This allows us to characterise the potential
$\phi:B(0,1)\to\R^d$ as the   supremum of affine functions
\begin{equation} 
 \label{p1}
\phi(y) =  \sup\limits_k \{ y\cdot d_k - v_k \}, 
\end{equation} 
which is convex by construction.  Determination of the OTM is reduced to 
the problem  of finding  real numbers $\{v_k\}$, the ``heights'' of the hyperplanes 
\bq\label{eq:phik}\phi_k(y) = \{y\cdot d_k - v_k\},\eq  
such that conservation of mass~\eqref{MAW} holds.
 Under the simplification~\eqref{s1} for $\nu$, this condition is equivalent to 
 \begin{equation} 
 \label{o1}
  | C_k | = \alpha _k ,\,\, \forall k  
 \end{equation} 
 where   
\bq\label{eq:cells}C_k  =\{ y \in B(0,1), \, T(y) =  d_k  \} \eq
is the support of the cell being  mapped to the Dirac at $d_k$.

\begin{definition}[Pogorelov solution]\label{def:pog}
We will refer to a function~(\ref{p1}) such that~(\ref{o1}) holds as a \emph{Pogorelov solution} of the OTM problem.
\end{definition}

Pogorelov provided an algorithm to determine each $v_k$  by remarking that lifting any of the hyperplanes $\phi_j$ (that is, decreasing $v_j$) results in increasing $|C_j|$ and decreasing all 
others values of $|C_k|$. Adjusting these heights sequentially until all volumes of the cells $\{C_k\}$ are correct converges in at most 
$\bO(N^2)$  iterations \cite{dot,Merigot}. 
Pogorelov construction has inspired many works~\cite{Cullen,olikerprussner88,Merigot} and remains the   state of the art  approach for numerically solving this particular 
instance of the semi-discrete Monge-Kantorovitch problem.
However, the method requires the computation of all
cells $C_k$ (also known as power diagrams) at each iteration, which requires $\bO(N \log N)$ operations~\cite{aurenhammer}.  
Thus the computational complexity is at worst $\bO(N^3 \log N)$, though it can be improved using multigrid techniques as proposed in~\cite{Merigot}.  \\

In our study, we explore   a different strategy based on   computing the 
Legendre-Fenchel dual $\phi^*$ of the potential $\phi$.  
While this function has a more complicated structure, we will show that  it is amenable 
to classical PDE approximation techniques on a grid.  This leads to a new method for solving the semi-discrete  OTM problem. 

We recall that the dual is defined as
\[
\phi^*(x) = \sup_{y \in B(0,1)} \{ x \cdot y - \phi(y) \} \]
where the potential 
$\phi$ can be extended to all space by assigning it the value $+\infty$ outside of  $B(0,1)$.

The dual $\phi^*$ can be understood as the potential associated with the inverse OTM.  That is, the subgradient $\partial\phi^*$ provides the optimal rearrangement between the atomic measure $\mu$ and the Lebesgue measure $\nu$ on the unit ball.  Additionally, we will show that $\phi^*$ can be interpreted as an Aleksandrov solution of the \MA equation.

The dual is intimately connected with the original OTM.  In particular, it satisfies the following properties (Theorem~\ref{thm:OTM}):
\[
\begin{cases}
\bar{C_k} = \partial\phi^*(d_k)\\
v_k = \phi^*(d_k).
\end{cases}
\]
Thus if $\phi^*$ is known, the original potential $\phi$ and mapping $\nabla\phi$ are easily constructed.

\subsection{Contents}
In \autoref{sec:dual}, we examine the properties of the Legendre-Fenchel dual $\phi^*$.  In particular, we show that $\phi^*$ solves a \MA equation in the Aleksandrov sense.  We also provide a geometric characterisation of this function.

In \autoref{sec:visc}, we introduce a modified viscosity formulation of the \MA equation, which is valid when the measure $\mu$ is given as a weighted sum of Dirac measures.  Away from the Diracs, this agrees with the usual definition of the viscosity solution.  At the Diracs, we provide a local formulation of the subgradient measure in terms of one-sided directional derivatives.  We show that this formulation is equivalent to the usual Aleksandrov formulation.
 
In \autoref{sec:scheme}, we describe a consistent, monotone approximation scheme for the modified viscosity formulation of the \MA equation.

Finally, in \autoref{sec:numerics}, we provide a numerical study that validates that correctness of our approach.  Comparison with a scheme designed for traditional viscosity solutions~\cite{BFO_OTNum} demonstrates that our modification is necessary.  While the results are low accuracy (approximately $\bO(\sqrt{h})$ in practice), the results could provide an excellent initialisation for more traditional Pogorelov based methods.

\section{The Legendre-Fenchel Dual} \label{sec:dual}
\subsection{Properties of the Legendre-Fenchel dual}

We now turn our attention to the Legendre-Fenchel dual $\phi^*$ of the OTM potential $\phi$.  We recall that  this is defined as
\[ \phi^*(x) = \sup\limits_{y\in B(0,1)}\{y\cdot x- \phi(y)\}. \]
We will also be interested in the subgradient of $\phi^*$, which defines a set-valued map from points in $\R^2$.  The subgradient at a point $p$ is defined as
 \begin{equation} 
 \label{sg2}
 \partial \phi^* (x) =  \{ p \in \R^2 \mid  \phi^*(z) \ge \phi^*(x) + p \cdot (z - x)\,\, \forall z\in \R^2\} 
 \end{equation}  
and the subgradient of a set $E$ is given by
 \begin{equation} 
  \label{sg1}
 \partial \phi^* (E) =  \bigcup\limits_{y \in E}  \partial \phi^*(y) .
 \end{equation} 

We list several basic properties of the Legendre-Fenchel transform~\cite[Section~2.2]{hormander}.

\begin{enumerate}[label={(P\arabic*)}]
\renewcommand{\labelenumi}{{\theenumi}}
\renewcommand{\theenumi}{(P\arabic{enumi})}
\item\label{prop:convex}  $\phi^*$ is  convex and lower semi-continuous.
\item\label{prop:envelope} The double Legendre transform $\phi^{**}$ is equal to the convex envelope of $\phi$.
\item\label{prop:grad} If $\phi$ is convex and lower semi-continuous, then $\phi^{**} = \phi$ and at points $y$ of differentiability~:  $y \in \partial\phi^*(\nabla\phi(y))$.
\item\label{prop:attained} If $\phi$ is convex and lower semi-continuous and $x\in\partial\phi(y)$, then 
$\phi(y)+\phi^*(x) = x\cdot y$.
\item\label{prop:affine} The dual of an affine function $\phi_k(y) = y\cdot d_k - v_k$, extended by $+\infty$ outside of the ball $B(0,1)$, is the cone $\phi_k^* = v_k + \|x-x_k\|$.
\item\label{prop:max} If $\{\phi_i(y)\}$ is a finite set of lower semi-continuous convex functions then the maximum of these functions $f = \max_i \{\phi_i(y) \mid i\in I\}$ is also a lower semi-continuous, convex function and its dual is given by
\[ \phi^*(x) = \inf\left\{\sum\limits_{k} \lambda_k \phi_k^*(x_k) \mid k \in I, \, x_k\in\R^2, \, x = \sum\limits_{k}\lambda_kx_k, \, \lambda_k \geq 0, \, \sum\limits_k \lambda_k = 1\right\}. \]
\end{enumerate}

Using these properties, we can deduce more about the relationship between the original OTM and the dual $\phi^*$.

\begin{theorem}[Relationship to OTM potential]\label{thm:OTM}
For each Dirac at $d_k$, the subgradient of $\phi^*$ gives the cell $C_k$,
\[ \partial\phi^*(d_k) = C_k, \]
and the height $v_k$ of each plane used to define the potential $\phi$ is given by 
\[ v_k = \phi^*(d_k). \]
\end{theorem}

\begin{proof}
To prove the first claim, we recall that by the definition of the cells $C_k$, if $y\in C_k$ then $\nabla \phi(y) = d_k$.  Then~\ref{prop:grad} yields
\[ C_k = \partial\phi^*(\nabla\phi(C_k)) = \partial\phi^*(d_k). \]

For the second part of this claim, we use that fact that by construction, $C_k$ has positive measure.  Thus we can find $y\in C_k$ with $\nabla\phi(y)=d_k$.  We recall that $\phi$ can be characterised as 
\[ \phi(y) = \max\limits_i\{y\cdot d_i - v_i\}, \]
which implies that $\phi(y) = y\cdot d_k - v_k$ and $\phi^*(d_k) = v_k$.
\end{proof}

\subsection{Aleskandrov solutions}

Next we show that the dual $\phi^*$ can be interpreted as a convex Aleksandrov solution of the \MA equation
\bq\label{eq:MA_Aleks}
\begin{cases}
\det(D^2 u) = \mu, & x\in\R^2\\
\partial u(\R^2) \subset B(0,1).
\end{cases}
\eq

\begin{definition}[Aleksandrov solution]\label{def:aleks}
A convex function $u$ is an \emph{Aleksandrov solution} of~\eqref{eq:MA_Aleks} if for every measurable set $E\in \R^2 $,
\[ \abs{\partial u(E)} = \mu(E). \]
\end{definition}

The main result of this section is the equivalence of Aleksandrov and Pogorelov solutions.

\begin{theorem}[Equivalence of Aleksandrov and Pogorelov solutions]\label{thm:AleksPog}
The function $u$ is an Aleksandrov solution of the \MA equation~\eqref{eq:MA_Aleks} if and only if it can be characterised as the Legendre-Fenchel dual $\phi^*$ of a Pogorelov solution $\phi$ of the \MA equation~\eqref{MA}.
\end{theorem}

\begin{proof}
The result follows immediately from Lemmas~\ref{lem:PogAleks}-\ref{lem:AleksPog}.
\end{proof}

\begin{lemma}[Pogorelov solutions are Aleksandrov]\label{lem:PogAleks}
Let $\phi$ be a Pogorelov solution of the \MA equation~\eqref{MA}.  Then $\phi^*$ is an Aleksandrov solution of the dual \MA equation~\eqref{eq:MA_Aleks}.
\end{lemma}

\begin{proof}
From~\ref{prop:convex} we deduce that $\phi^*$ is convex.

Next we show that the subgradient of $\phi^*$ is constrained to the closed unit ball.  Let $p\in\partial\phi^*(x)$.  Then for every $z\in\R^2$,
\begin{align*}
p\cdot(z-x) &\leq \phi^*(z)-\phi^*(x)\\
  &= \phi^*(z) - (y^*\cdot x - \phi(y^*)) & \text{for some $y^*\in B(0,1)$}\\
	&\leq (y^*\cdot z - \phi(y^*))-(y^*\cdot x - \phi(y^*)) \\
	&= y^*\cdot(z-x).
\end{align*}
In particular we can choose $z = x + p$, which yields
\[ \|p\|^2 = y^*\cdot p \leq \|p\|. \]
Thus $p\in B(0,1)$.

From Theorem~\ref{thm:OTM} we know that the subgradient measure is correct at each Dirac location $d_k$, $\abs{\partial\phi^*(d_k)} = \abs{C_k} = \alpha_k$.  This allows us to compute the subgradient measure of the whole space.
\[ \pi = \sum\limits_k\alpha_k \leq \abs{\partial\phi^*(\R^2)} \leq \abs{B(0,1)} = \pi. \]

It follows that the subgradient has zero measure away from the Diracs,
\[ \abs{\partial\phi^*\left(\R^2\backslash\bigcup\limits_k\{d_k\}\right)} = 0. \]

We conclude that the subgradient has the correct measure everywhere and satisfies the necessary constraint.  Thus $\phi^*$ is an Aleksandrov solution.
\end{proof}

\begin{lemma}[Aleksandrov solutions are Pogorelov]\label{lem:AleksPog}
Let $u$ be an Aleksandrov solution of the dual \MA equation~\eqref{eq:MA_Aleks}.  Then $u$ can be characterised as the dual $\phi^*$ of a function $\phi$ that satisfies the \MA equation~\eqref{MA} in the Pogorelov sense.
\end{lemma}

\begin{proof}
Let $u$ be a convex Aleksandrov solution and consider its dual
\[ u^*(y) = \sup\limits_{x\in\R^2}\{x\cdot y - u(x)\} \geq \max\limits_{k}\{d_k\cdot y - u(d_k)\}. \]
From~\ref{prop:convex}, $u^*$ is convex on the unit ball.

For almost every $y\in B(0,1)$ there exists some Dirac position $d_i$ such that $y\in\partial u(d_i)$.  As a consequence of~\ref{prop:attained}, we have 
\[ u^*(y) = d_i\cdot y - u(d_i). \]
We conclude that
\[ u^*(y) =  \max\limits_{k}\{d_k\cdot y - u(d_k)\} = \max\limits_{k}\{d_k\cdot y - v_k\}\]
where we have defined $v_k = u(d_k)$.

Next we define the cells $C_k = \{y\in B(0,1) \mid \nabla u^*(y) = d_k\}$.  Choose $y\in\partial u(d_k)$.  For almost every such $y$ we have $\nabla u^*(y) = d_k$; see~\ref{prop:grad}.  Thus
\[ \abs{C_k} \geq \abs{\partial u(d_k)} = \alpha_k. \]
Furthermore, we know that
\[ \sum\limits_k\abs{C_k} \leq \abs{B(0,1)} = \pi = \sum\limits_k\alpha_k. \]
We conclude that $\abs{C_k}= \alpha_k$ and thus $u^*$ is a Pogorelov solution of~\eqref{MA}.

By~\ref{prop:envelope}, $u=u^{**}$ so that $u$ is the dual of a Pogorelov solution.
\end{proof}

Finally, we make the claim that the subgradient map takes values outside the convex hull of the Dirac points $d_k$ into the boundary of the unit ball.

\begin{lemma}[Map onto the boundary]\label{lem:bdyToBdy}
Let $x$ be any point outside the convex hull of the Dirac points $\{d_k\}$.  Then $\partial\phi^*(x) \in \partial B(0,1)$.
\end{lemma}

\begin{proof}
Let $\Omega$ be the convex hull of the Dirac points.  From Lemma~\ref{lem:PogAleks}, $\partial\phi^*(\Omega)\in B(0,1)$ and $\abs{\partial\phi^*(\Omega)} = \pi$.  Thus the subgradient maps $\Omega$ onto the closed unit ball,
\[ \partial\phi^*(\Omega) = B(0,1). \]
Now choose any $x\notin\Omega$.  From Lemma~\ref{lem:PogAleks}, $\partial\phi^*(x) \in B(0,1)$.  Then the cyclical monotonicity of the map $\partial\phi^*$ reqires that $p\in\partial B(0,1)$.
\end{proof}

\subsection{Geometric characterisation}

We now provide a geometric characterisation of the dual $\phi^*$.  

\begin{lemma}[Characterisation as convex envelope]\label{lem:envelope}
Define the function
\[ \psi(x) = \min\limits_k\{\phi^*(d_k) + \|x-d_k\|\}. \]
Then $\phi^*=\psi^{**}$ is the convex envelope of $\psi$.
\end{lemma}

\begin{proof}
We consider the cones
\[ \psi_k(x) = \phi^*(d_k) + \|x-d_k\|. \]
From~\ref{prop:affine} this is the dual $\phi_k^*$ of the plane
\[ \phi_k(y) = y\cdot d_k - \phi^*(d_k). \]
That is, $\psi_k = \phi_k^*$.  Applying~\ref{prop:max} we obtain
\bq\label{eq:phidual} \phi^*(x) = \inf\left\{\sum\limits_k\lambda_k\psi_k(x_k) \mid k \in I, \, x_k\in\R^2, \, x = \sum\limits_{k}\lambda_kx_k, \, \lambda_k \geq 0, \, \sum\limits_k \lambda_k = 1\right\} \eq
and $\phi^*$ is convex.

Now we fix any $x\in\R^2$.  For some $j$ we have $\psi(x) = \psi_j(x)$.  Taking $\lambda_j=1$ in the above characterisation we find that $\phi^*(x) \leq \psi(x)$.

Suppose that $\phi^*$ is not the convex envelope of $\psi$.  Then there exists a convex function $v \leq \psi$ such that $v(x) > \phi^*(x)$ for some $x\in\R^2$.  For any $x_k,\lambda_k$ satisfying the constraints
\[ k \in I, \, x_k\in\R^2, \, x = \sum\limits_{k}\lambda_kx_k, \, \lambda_k \geq 0, \, \sum\limits_k \lambda_k = 1 \]
we have
\[ v(x) \leq \sum\limits_k \lambda_kv(x_k) \leq \sum\limits_k \lambda_k\psi(x_k) \leq \sum\limits_k\lambda_k\psi_k(x_k). \]
Computing the infinum over all $x_k,\lambda_k$ in the constraint set, we find that $v(x) \leq \phi^*(x)$, a contradiction.
\end{proof}

This provides a nice geometrical characterisation of the graph of $\phi^*$, which is obtained by computing the infimum of the cones $\phi_k^*$ centred at the Dirac masses, then constructing the convex envelope of the result.  The result agrees with each cone at the corresponding Dirac location $d_k$ (the tip of the cone) and the height coincides with the $v_k$ required for the original OTM.

\section{Mixed Aleksandrov-Viscosity Formulation}\label{sec:visc}
The dual $\phi^*$ can be properly understood as the Aleksandrov solution of a \MA equation.  However, there is no simple framework available for the numerical approximation of these solutions.  

An alternative notion of weak solution is the viscosity solutions.  A rich theory is available for viscosity solutions of degenerate elliptic equations~\cite{CIL}, including the construction and convergence of approximation schemes~\cite{BS_Approx,ObermanDiffSchemes}.  Unfortunately, the usual viscosity theory cannot accommodate the case where the source is a sum of Dirac measures.

In this section, 
we develop a modified notion of viscosity solutions that allows for a source measure of the form
\[ \mu = \sum\limits_k \delta_k. \]

\subsection{Viscosity solutions}\label{sec:viscOrig}

We begin by reviewing the conventional definition of viscosity solutions.

When the given measure $\mu$ can be expressed in terms of a density function $f(x)$, 
\[ \mu(E) = \int\limits_E f(x)\,dx,  \]
we can define viscosity solutions of the \MA equation.  The essence of this type of weak solution is the comparison principle, which allows us to move derivatives off of potentially non-smooth solutions and onto smooth test functions.

Before we give the definition of the viscosity solution, it is important to write the \MA operator in an appropriate form.  Several points need to be addressed:
\begin{enumerate}
\item The state constraint $\partial u(X) \subset B(0,1)$ must be included.  As in~\cite{BFO_OTNum}, we express this condition as a Hamilton-Jacobi equation,
\[ H(\nabla u(x)) = 0, \quad x\notin X, \]
where $H(y)$ is the signed-distance to the boundary of the unit ball $B(0,1)$ and $X$ is any convex set that contains the support of the source density $f$.
Comparison principles for degenerate elliptic equations complemented by these boundary conditions can be found in \cite{barlesbc}.
\item The solution of the \MA equation is only unique up to an additive constant.  To fix a unique solution, we search for the solution that has mean zero.  This condition can be incorporated into the equation by adding $\langle u\rangle$ (the mean of the solution $u$) to the boundary operator.
\item The viscosity solution framework only holds for degenerate elliptic equations. The Monge-Amp\`ere equation is 
degenerate elliptic only on the set of convex function. In Appendix~\ref{app:convexity}, we recall a solution proposed by one of the authors for incorporating the 
convexity constraint into the Monge-Amp\`ere operator.  
\end{enumerate}

This allows us to express the \MA operator as
\bq\label{eq:MA_visc1}
F(x,u(\cdot),\nabla\phi(x),D^2u(x)) \equiv \begin{cases}
-\det(D^2u(x)) + f(x), & x\in X \\
H(\nabla u(x))-\langle u \rangle, & x\in\partial X .
\end{cases}
\eq


The definition of the viscosity solution relies on the value of the operator applied to smooth test functions $\phi$.  In particular, we need to rely on the value of the operator at points where $u-\phi$ has a local extremum.  In order to account for the equation at boundary points, we slightly abuse the usual notion of an extremum as follows.

\begin{definition}[Extremum]
Let $\phi\in C^2$.  We say that $u-\phi$ has a local maximum (minimum) at the point $x_0\in\partial X$ if $u(x_0) - \phi(x_0) \geq (\leq) u(x)-\phi(x)$ in a neighbourhood of $x_0$ and, additionally, there exists a sequence $x_n\notin\partial X$, $\phi_n(x)\in C^2$  with $(x_n,\phi_n(x_n),\nabla\phi_n(x_n),D^2\phi_n(x_n)) \to (x_0,\phi(x_0),\nabla\phi(x_0),D^2\phi(x_0))$ such that  $u-\phi_n$ has a local maximum (minimum) at the point $x_n$.
\end{definition}

\begin{definition}[Viscosity solution]
A convex function $u$ is a \emph{viscosity subsolution (supersolution)} of~\eqref{eq:MA_visc1} if for every convex function $\phi\in C^2$, if $u-\phi$ has a local maximum (minimum) at $x_0$ then 
\[ 
F(x_0,u(\cdot),\nabla\phi(x_0),D^2\phi(x_0)) \leq(\geq)  0 .
\]
A convex function $u$ is a \emph{viscosity solution} if it is both a subsolution and a supersolution.
\end{definition}

\subsection{Mixed Aleksandrov-viscosity solutions}

We now introduce a new formulation of the \MA equation that combines the definitions of viscosity and Aleksandrov solutions.
Note that the notion of viscosity solution (\ref{eq:MA_visc1}) is local and holds for $ f(x) = 0$. 
When the source is given by
\[ \mu = \sum_{k=1}^K \alpha_{k} \delta_{d_k}, \]
we show that is is sufficient to modify the operator at the Dirac points. 


For a convex function $u$, we now define the \MA operator by
\bq\label{eq:MA_visc}
F(x,u(\cdot),\nabla u(x), D^2 u(x)) = 
\begin{cases}
-\det(D^2u(x)) , & x\in X \backslash\bigcup\limits_{k=1}^K\{d_k\}\\
-M[u](d_k) + \alpha_k, & k=1,\ldots,K\\
H(\nabla u(x)) - \langle u \rangle , & x\in\partial X .
\end{cases}
\eq


Then we can define viscosity solutions of~\eqref{eq:MA_visc} as follows.

\begin{definition}[Mixed Aleksandrov-viscosity solution]\label{def:mixed}
A convex function $u$ is a \emph{subsolution (supersolution)} of~\eqref{eq:MA_visc} if 
\begin{enumerate}
\item For every $x_0 \notin \bigcup\limits_k\{d_k\}$ and smooth convex function $\phi$, if $u-\phi$ has a local maximum (minimum) at $x_0$ then 
\[ 
F(x_0,u(\cdot),\nabla\phi(x_0),D^2\phi(x_0)) \leq(\geq)  0.
\]
\item $-M[u](d_k) +\alpha_k\geq (\leq) 0$.
\end{enumerate}
A convex function $u$ is a \emph{mixed Aleksandrov-viscosity solution} if it is both a subsolution and a supersolution.
\end{definition}


We claim that this modified notion of viscosity solutions is consistent with the notion of the Aleksandrov solution.
 
\begin{theorem}[Equivalence of solutions]\label{thm:ViscAleks}
A convex function $u$ with mean zero is a viscosity solution of~\eqref{eq:MA_visc} if and only if it is an Aleksandrov solution of~\eqref{eq:MA_Aleks} in $X$.
\end{theorem}

\begin{proof}
The proof follows immediately from Lemma~\ref{lem:AleksVisc} and Lemma~\ref{lem:ViscAleks}.
\end{proof}

\begin{lemma}\label{lem:AleksVisc}
Let $u$ be a convex Aleksandrov solution of~\eqref{eq:MA_Aleks} with mean zero.  Then $u$ is a mixed Aleksandrov-viscosity solution of~\eqref{eq:MA_visc}.
\end{lemma}
\begin{proof}
We first demonstrate that $u$ is a subsolution of~\eqref{eq:MA_visc}.  

Let $\phi\in C^2$ be convex and suppose that $u-\phi$ has a maximum at $x_0$, which we can take to be a global maximum without loss of generality.  

If $x_0\in X\backslash\bigcup\limits_{k=1}^K d_k$, then the proof proceeds as in Proposition~1.3.4 of~\cite{Gutierrez}.

Next consider, $x_0 \in \partial X$.  Since $u$ is convex and $\phi$ is smooth, this situation can only arise if $\nabla u(x_0)$ exists.  The fact that $u-\phi$ has an extremum requires that 
\[ \nabla\phi(x_0) = \nabla u(x_0) \subset \partial Y. \]
Then from the definition of the signed distance function $H$ we have
\[ H(\nabla\phi(x_0)) = 0. \]

Finally, we note that since $u$ is an Aleksandrov solution, $-M[u](d_k)+\alpha_k = -\abs{\partial u(d_k)}+\alpha_k = 0 $.  Thus $u$ is a subsolution.

Next we demonstrate that $u$ is a supersolution.  At points $x_0\in X$, the proof proceeds as before.  We now consider $x \in\partial X$, $\phi\in C^2$, and suppose that $u-\phi$ has a minimum at $x_0$.  This requires $\nabla\phi(x_0)\in\partial u(x_0)\subset\partial Y$.  As before we have
\[ H(\nabla\phi(x_0)) = 0. \qedhere\]
\end{proof}
 
\begin{lemma}\label{lem:ViscAleks}
Let  a convex function $u$ be a mixed Aleksandrov-viscosity solution of~\eqref{eq:MA_visc}.  Then $u$ is an Aleksandrov solution of~\eqref{eq:MA_Aleks} in $X$.
\end{lemma}

\begin{proof}
From the definition of the mixed Aleksandrov-viscosity solution, the measure is correct at the Dirac points,
\[ \abs{\partial u(d_k)} = \alpha_k. \]

At points $x\in X\backslash\bigcup\limits_k\{d_k\}$, the proof of~\cite[Theorem~1.7.1]{Gutierrez} is easily adapted to the case $f=0$ to show that 
\[ \abs{\partial u\left(X\backslash\bigcup\limits_k\{d_k\}\right)} = 0. \]

Finally, we need to show that $\partial u(X) \subset B(0,1)$.  We begin by defining a ball of radius $1+\langle u \rangle$: $\tilde{B} = B(0,1+\langle u \rangle)$.  Note that the boundary of this ball corresponds to the set $\{p \mid H(p)-\langle u \rangle = 0\}$.

Choose any $x_0\in\partial X$ and $p\in\partial u(x_0)$.  Then the plane $\phi(x) = p\cdot(x-x_0)$ is such that $u-\phi$ has a minimum at $x_0$.  Thus $H(p)-\langle u \rangle \geq 0$ so that $p \notin \tilde{B}^\circ$, the interior of the ball.

Now by convexity we know that $\nabla u(x_0)$ exists for almost every $x_0\in\partial X$ so that $H(p)-\langle u \rangle \leq 0$ at these points.  Combined with the previous line, we can conclude that for almost every $x_0\in\partial X$, $\partial u(x_0) \in \partial\tilde{B}$.  By convexity of $u$ and the fact that no element of $\partial u(\partial X)$ lies in the interior of $\tilde{B}$, we can conclude that $\partial u(\partial X) = \partial\tilde{B}$ and $\partial u(X) = \tilde{B}$.

We can now compute the subgradient measure over the domain as $\abs{\partial u(X)} = \abs{\tilde{B}}$.  However, we also know that
\[ \abs{\partial u(X)} = \sum\limits_k\alpha_k = \abs{B(0,1)}. \]
The two balls $B(0,1)$ and $B(0,1+\langle u \rangle)$ must then have the same area so that $\langle u \rangle = 0$ and $\partial u(X) \subset B(0,1)$.
\end{proof}

\subsection{Characterisation of subgradient measure}\label{sec:subgradient}

We introduce an alternative characterisation of the subgradient measure of a convex function.  
This approach involves rewriting the subgradient at a point on a convex function  in angular terms using  the one-sided directional derivative $\partial_\theta$ defined as
\bq\label{eq:onesided}
\partial_\theta  u (x) = \lim_{r \rightarrow 0+} \dfrac{u(x+r\, e_\theta) - u(x)}{r} 
\eq
where $e_\theta  = (\cos \theta  , \sin \theta) $.


\begin{lemma}
Let $u$ be a convex function.  Then for every $x$ in its domain and every $\theta\in[0,2\pi)$, the one-sided derivative $\partial_\theta u(x)$ is defined.  Moreover, the one-sided derivatives are continuous in $\theta$.
\end{lemma}

\begin{proof}
The function $u$ is convex and therefore Lipschitz continuous.  Thus one-sided directional derivatives are continuous as in~\cite{Shapiro}.
\end{proof}

This allows us to use a local characterisation to rewrite the subgradient measure of a convex function at a point $x$ in its domain.

\begin{theorem}[Characterisation of subgradient measure]\label{thm:subgradient}
Let $u$ be a convex function.  Then at any point $x$ in its domain, the subgradient measure is equivalent to
\bq\label{eq:subgradmeasure}
\abs{\partial u (x)} = M[u](x) \equiv
  \int_0^{2\pi} \frac{1}{2}\left(R_+(x,\theta)^2 - R_-(x,\theta)^2\right)^+\,d\theta,
\eq
where we define
\bq\label{eq:Bp}
R_+(x,\theta) =  \inf_{  \theta' \in ( \theta-\frac{\pi}{2},  \theta+\frac{\pi}{2} )} 
  \frac{(\partial_{\theta'} u(x))^+}{ \cos(\theta - \theta')},
\eq
\bq\label{eq:Bm}
R_-(x,\theta) =  
 \sup_{  \theta' \in ( \theta-\frac{\pi}{2},  \theta+\frac{\pi}{2} )} \frac{(-\partial_{\theta'+\pi} u(x))^+ }{\cos(\theta - \theta') } .
\eq
\end{theorem}

\begin{proof}
If $u$ is convex, the subgradient can be expressed as
\[ 
\partial u (x) = \left\{ y \in \R^2 \mid  y \cdot e_\theta \le \partial_\theta u(x) \, \forall \theta \in [0 ,2\pi)\right\},
\]
which is convex and compact set. 

We reformulate the subgradient measure in polar coordinates ($y\rightarrow (r,\theta')$). 
\begin{align*}
\partial u (x) & = \left\{ (r,\theta')\in \R_+\times [0,2\pi) \mid  r \, e_{\theta'}\cdot e_\theta \le  \partial_{\theta} u(x) , \, \, \forall \theta \in [0,2\pi)\right\} \\
& = \left\{ (r,\theta')\in \R_+\times [0,2\pi] \mid  r \,  \cos(\theta' - \theta)  \le  \partial_{\theta} u(x), \, \, \forall \theta \in [0,2\pi)\right\}  .
\end{align*} 

When $\theta \neq \theta'\pm\frac{\pi}{2}$, we can divide through by $\cos(\theta'-\theta)$ in the constraint.  Because of the continuity of the one-sided derivatives, it is sufficient to check the constraint for values of $\theta$ in the open set $(\theta'-\frac{\pi}{2},  \theta'+\frac{\pi}{2})$.  Thus the subgradient can be rewritten as
\[\partial u (x) = \left\{ (r,\theta')\in \R_+\times [0,2\pi) \mid  \frac{-\partial_{\theta+\pi} u(x) }{\cos(\theta' - \theta) }    \le  r    \le  
  \frac{\partial_{\theta} u(x)}{ \cos(\theta' - \theta)}   \, \, \forall \, \theta \in \left( \theta'-\frac{\pi}{2},  \theta'+\frac{\pi}{2}\right) \right\}  .
\]
Since $r$ must be positive, this is equivalent to
\[
\partial u (x) = \left\{ (r,\theta')\in \R_+\times [0,2\pi) \mid  \frac{(-\partial_{\theta+\pi} u(x))^+ }{\cos(\theta' - \theta) }    \le  r    \le  
  \frac{(\partial_{\theta} u(x))^+}{ \cos(\theta' - \theta)}   \, \, \forall\, \theta \in \left( \theta'-\frac{\pi}{2},  \theta'+\frac{\pi}{2}\right) \right\}  .
\]

Finally, by defining
 \[R_+(x,\theta') =  \inf_{  \theta \in ( \theta'-\frac{\pi}{2},  \theta'+\frac{\pi}{2} )} 
  \frac{(\partial_{\theta} u(x))^+}{ \cos(\theta' - \theta)},  \] 
	\[R_-(x,\theta') =  
 \sup_{  \theta \in ( \theta'-\frac{\pi}{2},  \theta'+\frac{\pi}{2} )} \frac{(-\partial_{\theta+\pi} u(x))^+ }{\cos(\theta' - \theta) }, \]
we can rewrite the subgradient as
\[
\partial u (x) = \left\{ (r,\theta')\in \R_+\times [0,2\pi) \mid  
 R_-(x,\theta')    \le  r    \le 
  R_+(x,\theta')  \right\} . \]

 We can integrate over this region, noticing that only values that satisfy $R_- \le R_+$ will contribute to the integral. This allows us to replace $R_-$ by $ \min\{R_-,R_+\} $ to avoid 
 negative values. 
 Then the measure of the subgradient can be expressed as
 
\begin{align*} 
\label{SGP}
|\partial u (x)| &= \int_{0}^{2\pi} \int_{\min(R_-(x,\theta'),R_+(x,\theta'))}^{ R_+(x,\theta') } r \, dr \, d\theta'\\
  &= 
\int_{0}^{2\pi}  \frac{1}{2} (R_+(x,\theta')^2-\min(R_-(x,\theta'),R_+(x,\theta'))^2 )\, d\theta'\\
 &= \int_0^{2\pi} \frac{1}{2}\left(R_+(x,\theta')^2 - R_-(x,\theta')^2\right)^+\,d\theta'.
\end{align*} 
Interchanging the roles of $\theta$ and $\theta'$, we recover the claim of Theorem~\ref{thm:subgradient}.
\end{proof}

\section{Approximation Scheme}\label{sec:scheme}
We now turn our attention to the construction of an approximation scheme for the \MA equation.  

\subsection{Monge-Amp\`ere operator and boundary conditions}
One of the challenges in numerically solving the \MA equation is enforcing convexity of the solution.  We follow the approach of~\cite{FroeseTransport} and replace the \MA operator by a ``convexified'' operator,
\[-\min\limits_{\abs{\nu} = 1, \nu\cdot\nu^\perp = 0}\left\{u_{\nu\nu}^+u_{\nu^\perp\nu^\perp}^+-u_{\nu\nu}^- - u_{\nu^\perp\nu^\perp}^-\right\}, \quad x\in X \backslash\bigcup\limits_{k=1}^K\{d_k\},\]
which encodes the convexity constraint.

A mixed Aleksandrov-viscosity solution can be described for this modified operator.  See Appendix~\ref{app:convexity} for a proof that the resulting solution is equivalent to the Aleksandrov solution.

The schemes used for the basic \MA operator and the Hamilton-Jacobi boundary conditions have been thoroughly described in~\cite{BFO_OTNum,FO_MATheory}.  We briefly summarise the scheme (Appendix~\ref{app:MA}) and turn our attention
to the discretisation used at the dirac points $d_k$. 

\subsection{Discretisation of subgradient measure}\label{sec:discMeasure}

We constructin this section  a monotone approximation of the subgradient operator at the Dirac points.

We use  an angular discretisation $\theta_i, \, i=1,\ldots,N$ of the periodic interval $[0,2\pi)$. This does not need to be a uniform discretisation, but is instead chosen so that whenever the point $x$ lives on the grid, there exists a value $l_i>0$ such that $x\pm l_ihs_{\theta_i}$ also lies on the grid. 
 The resulting angular resolution is
\[d\theta = \max\limits_{1 \leq i \leq N} d\theta_i = \max\limits_{1 \leq i \leq N}\left\{\frac{1}{2}(\theta_{i+1}-\theta_{i-1})\right\}.\]
This allows us to approximate one-sided directional derivatives at the Dirac location $d_k$ by
\bq\label{eq:onesided} \partial_{\theta_i} u(d_k)  \simeq \frac{u(d_k+l_ihe_{\theta_i}) - u(d_k)}{l_ih} \eq
where we define
\[ u_i = u(d_k+l_ihe_{\theta_i}), \quad u_{-i} = u(d_k+l_ihe_{\theta_i+\pi}), \quad u_0 = u(d_k). \]

%

This allows $R_+$ and $R_-$ to be approximated by
\begin{equation} 
\label{SGPD}
\begin{array} {l} 
{R}^\rho_+(d_k,\theta_i) =  \inf\limits_{  \theta_j \in ( \theta_i-\frac{\pi}{2},  \theta_i+\frac{\pi}{2} )} 
  \frac{(u_j -  u_0  )^+}{l_jh\, \cos(\theta_i- \theta_j)}  \\
  {R}^\rho_-(d_k,\theta_i) =  \inf\limits_{  \theta_j \in ( \theta_i-\frac{\pi}{2},  \theta_i+\frac{\pi}{2} )} 
  \frac{(u_0 -  u_{-j}  )^+}{ l_jh\, \cos(\theta_i- \theta_j)} .
\end{array} 
\end{equation} 

Finally, we  discretise the integral : 
\begin{equation} 
\label{eq:subgradDisc}
M[ u ](d_k) \simeq   \sum_{i=1}^{N} \frac{1}{2}d\theta_i \left(R^\rho_+(d_k,\theta_i)^2-R^\rho_-(d_k,\theta_i)^2 \right)^+.
\end{equation}

We can write the combined approximation scheme as 
\[ F^\rho(x,u(x),u(\cdot)) = 0 \]
where $\rho$ encodes the discretisation parameters $h, d\theta$.  For consistency, we need to choose $d\theta=d\theta(h)$ so that $d\theta, h/d\theta\to0$ as $h\to0$. 

\subsection{Properties of the approximation scheme}

We establish several important properties of the approximation scheme including consistency, monotonicity, and stability of the solutions.  


\begin{lemma}[Consistency]\label{lem:consistent}
Let $x\in\R^2$, $\phi\in\test$, and $\rho>0$.  Then as $\rho\to0$, the approximation scheme
\[ F^\rho(x,\phi(x),\phi(\cdot)) \to F(x,\phi(x),\nabla\phi(x),D^2\phi(x)). \]
\end{lemma}
\begin{proof}
At Dirac points $d_k$, the scheme is approximated by replacing one-sided derivatives with a first-order accurate approximation.  The approximation of the integral  is also consistent.  At other points, derivatives are replaced by consistent approximations as established in~\cite{BFO_OTNum,FO_MATheory}.
\end{proof}

\begin{lemma}[Monotonicity]\label{lem:monotonicity}
For every $\rho>0, x\in\bar{ X }, s\in\R$, and bounded $u \leq v$ we have
\[ F^\rho(x,s,u) \geq F^\rho(x,s,v). \]
\end{lemma}
\begin{proof}
The monotonicity of the discretisation of the \MA operator and Hamilton-Jacobi equation have been established in~\cite{BFO_OTNum,FO_MATheory}, so we need only check monotonicity of the approximation of the subgradient measure~\eqref{eq:subgradDisc}.

Notice that the discrete version of 
\[
R^\rho_+(d_k,\theta_i) =  \inf\limits_{  \theta_j \in ( \theta_i-\frac{\pi}{2},  \theta_i+\frac{\pi}{2} )} 
  \frac{(u_j -  u_0  )^+}{l_jh\, \cos(\theta_i- \theta_j)}  \\
\] 
is a decreasing function of the value $u_0 = u(d_k)$ and an increasing function of the remaining function values $u_j$.  Since it vanishes for $u_j-u_0<0$, the squared term $(R^\rho_+)^2$ inherits the same monotonicity properties.  Similarily, $-(R^\rho_-)^2$ satisfies the appropriate monotonicity conditions.  Since the discrete version of the integral is a monotone combination of these,
\[M[u] (d_k) \simeq   \sum_{i=1}^{N} \frac{1}{2}d\theta_i \left(R^\rho_+(d_k,\theta_i)^2-R^\rho_-(d_k,\theta_i)^2 \right)^+,
  \]
the scheme is monotone.
\end{proof}

\begin{lemma}[Stability]\label{lem:stability}
For every $\rho>0$, the scheme $F^\rho(x,u(x),u(\cdot))=0$ has a unique solution $u^\rho(x)$, which is bounded in $\ell^\infty$ independently of $\rho$.
\end{lemma}
\begin{proof}
The stability of degenerate elliptic (monotone) schemes was established by Oberman in~\cite{ObermanDiffSchemes}.
\end{proof}

\section{Numerical results} \label{sec:numerics}

We now provide several numerical tests of our formulation.  We use the following parameters in our computations.
\begin{itemize}
\item $N_X$: The number of grid points along each dimension of the domain.
\item $h$: The spatial stepsize along each dimension.
\item $w$: The maximum stencil width used in each direction, which is chosen to be $\floor{1/\sqrt{dx}}$.  The same stencil width is used for schemes at both the Dirac and non-Dirac points.
\item $N_Y$: The number of directions used to discretise the target set in the Hamilton-Jacobi equation, which is taken to be $4N_X$.
\end{itemize}

The discretised system of equations is solved using Newton's method.

\subsection{Comparison to viscosity solver}\label{sec:resultsVisc}

In our first example, we consider the problem of mapping a single Dirac, located at the origin and with total mass $\pi$, onto the unit circle.  We pose this problem in the domain $[-1,1]\times[-1,1]$.  The exact potential (up to an additive constant) is
\[ u_{ex}(x,y) = \sqrt{x^2+y^2}, \]
which is pictured in Figure~\ref{fig:OneDiracSol}.

We compute the solution in two ways:
\begin{enumerate}
\item Using the mixed Aleksandrov-viscosity formulation described in this paper.
\item Using a scheme designed for viscosity solutions~\cite{BFO_OTNum}, where the density at the Dirac mass is set to be
\[ f = \frac{4}{(w^2+(w-1)^2)h^2}. \]
\end{enumerate}

To evaluate the error, we shift the computed solution and the exact solution to have mean zero, then look at the maximum difference between the two solutions.  The resulting error is presented in Figure~\ref{fig:OneDiracError} and Table~\ref{table:OneDirac}.  In this radially symmetric example, both methods produce results that have first-order accuracy in the spatial resolution $h$.

We repeat the above experiment, this time placing two equally weighted masses in the domain at positions $(-0.5,0)$ and $(0.5,0)$.  In this case, the exact solution (up to an additive constant) is given by
\[ u_{ex}(x,y) = \begin{cases}  
\min\{\sqrt{(x+0.5)^2+y^2},\sqrt{(x-0.5)^2+y^2}\}, & \abs{x}>0.5\\
\abs{y}, & \text{otherwise}.
\end{cases}\]

The results are presented in Figure~\ref{fig:TwoDiracSol}-\ref{fig:TwoDiracError} and Table~\ref{table:OneDirac}.  Without the radial symmetry of the previous example, the viscosity formulation does not appear to the correct solution.  The Aleksandrov-viscosity scheme, on the other hand, demonstrates convergence with an accuracy of approximately $\bO(\sqrt{h})$.

\begin{figure}[htdp]
	\centering
	\subfigure[]{\includegraphics[width=.45\textwidth]{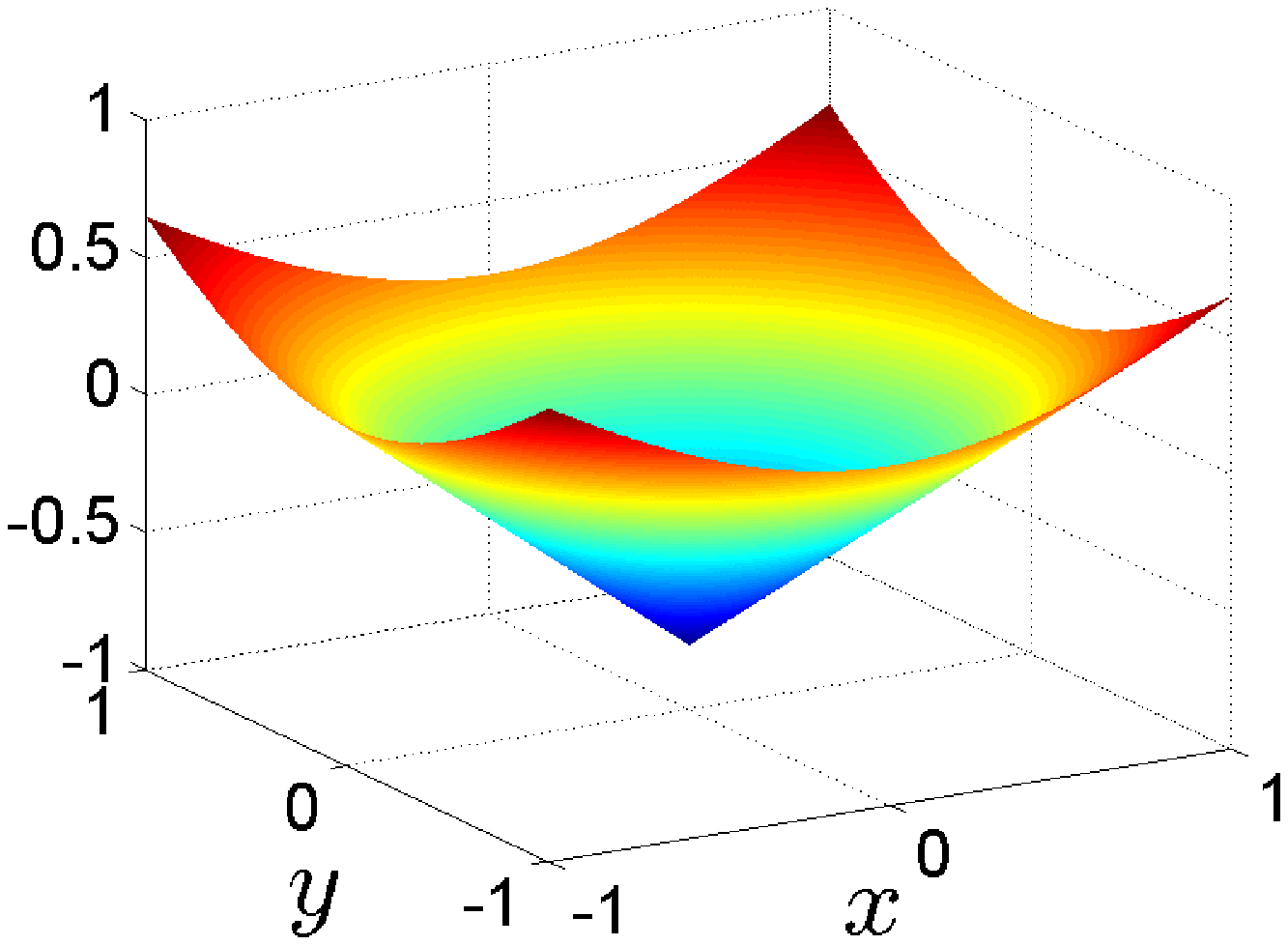}\label{fig:OneDiracSol}}
  \subfigure[]{\includegraphics[width=.45\textwidth]{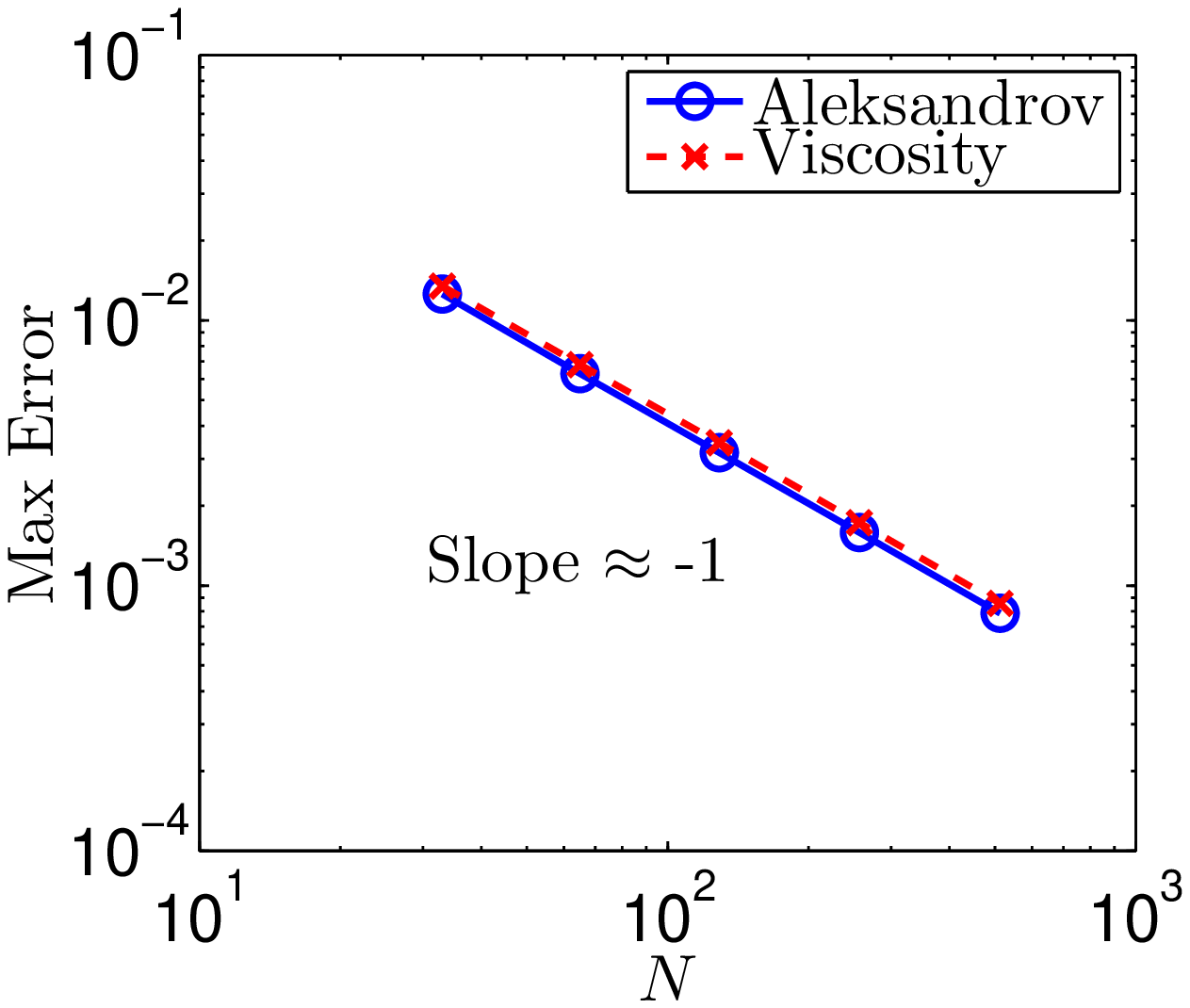}\label{fig:OneDiracError}}
  \subfigure[]{\includegraphics[width=.45\textwidth]{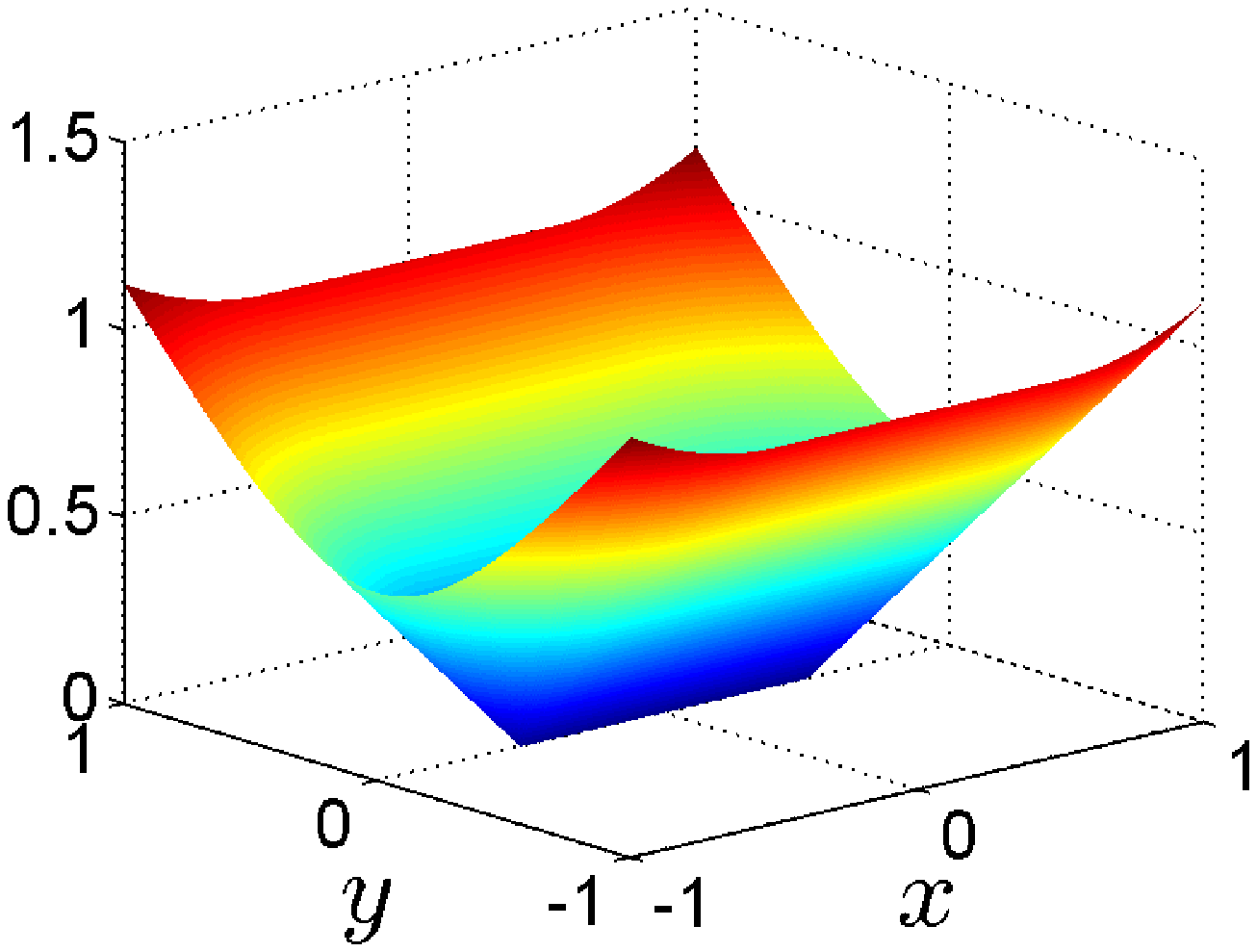}\label{fig:TwoDiracSol}}
  \subfigure[]{\includegraphics[width=.45\textwidth]{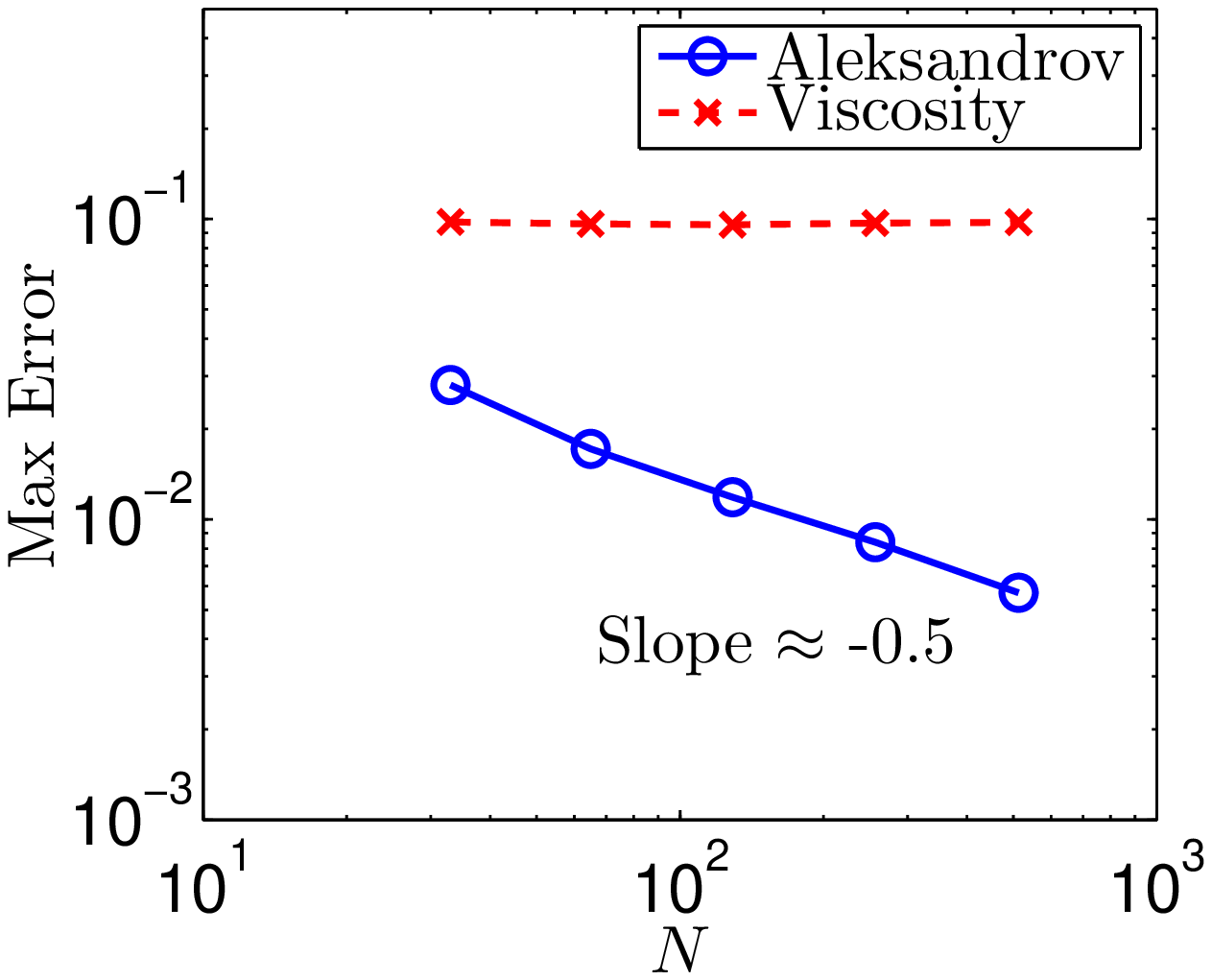}\label{fig:TwoDiracError}}
  	\caption{\subref{fig:OneDiracSol}~The potential for a single Dirac mass and \subref{fig:OneDiracError}~the maximum error in the computed solution for the Aleksandrov-viscosity and viscosity schemes
		\subref{fig:TwoDiracSol}~The potential for two Dirac masses and \subref{fig:TwoDiracError}~the maximum error in the computed solution for the Aleksandrov-viscosity scheme.}
  	\label{fig:Dirac}
\end{figure}

\begin{table}[htdp]\small
\begin{center}
\begin{tabular}{c|cc|cc}
$N_X$ &\multicolumn{4}{c}{Maximum Error} \\
  & \multicolumn{2}{c}{One Dirac} & \multicolumn{2}{c}{Two Diracs}\\
  & Aleksandrov-viscosity & Viscosity & Aleksandrov-viscosity & Viscosity\\
\hline
33  & $12.56\times10^{-3}$ & $13.44\times10^{-3}$ & $2.80\times10^{-2}$ & $9.78\times10^{-2}$ \\
65  & $6.29\times10^{-3}$ & $6.80\times10^{-3}$ & $1.72\times10^{-2}$ & $9.63\times10^{-2}$\\
129 & $3.17\times10^{-3}$ & $3.45\times10^{-3}$ & $1.12\times10^{-2}$ & $9.56\times10^{-2}$\\
257 & $1.58\times10^{-3}$ & $1.72\times10^{-3}$ & $0.84\times10^{-2}$ & $9.69\times10^{-2}$\\
513 & $0.79\times10^{-3}$ & $0.86\times10^{-3}$ & $0.57\times10^{-2}$ &$9.75\times10^{-2}$\\
\end{tabular}
\end{center}
\caption{Maximum error in the computed solution for the Aleksandrov-viscosity and viscosity schemes with one or two Dirac masses.}
\label{table:OneDirac}
\end{table}

\subsection{Comparison to exact solver}\label{sec:resultsFew}

In the next section, we set the location $d_k$ of several Dirac masses as well as the value of the potential $u_{ex}(d_k)$ at these masses.  
Then we use an exact construction of the cells $|C_k|$ provided by the Multi-Parametric Toolbox (MPT) for Matlab~\cite{MPT}.
 Finally, we use our Aleksandrov-viscosity method to solve the \MA equation with these weights and evaluate the maximum error in the potential at a Dirac mass,
\[ \max\limits_k\{\abs{u(d_k)-u_{ex}(d_k)}\}. \]

In our first example, we set three Dirac masses at
\[ d_1 = (-0.5,-0.5), \, d_2 = (0.5,-0.5), \, d_3 = (0.5,0.5) \]
with weights
\[ \alpha_1 =  1.17810586, \, \alpha_2 = 0.78540476,\, \alpha_3 = 1.17810586.\]
With these values, the potential $u$ should have the same value at each of the Dirac masses.

We repeat this example using five Dirac masses at
\[ d_1 = (0.5,0.5), \, d_2 = (0.5,-0.5), \, d_3 = (-0.5,0.5), \, d_4 = (-0.5,-0.5), \, d_5 = (0,0)\]
with weights
\[ \alpha_1 = 0.70539704, \, \alpha_2 =  0.56674540, \, \alpha_3 = 0.56674541,\]
\[ \alpha_4 = 1.142723415, \, \alpha_5 = 0.16000240. \]
At the Dirac masses, the potential should take on the values (up to a constant shift)
\[ u_{ex}(d_1) = u_{ex}(d_2) = u_{ex}(d_3) = 1, \, u_{ex}(d_4) = u_{ex}(d_5) = 0.8. \]

Finally, we place ten Dirac masses at
\[ d_1 = (0.5,0.5), \, d_2 = (0.5,-0.5), \, d_3 = (-0.5,0.5), \, d_4 = (-0.5,0.5), \]
\[d_5 = (0.25,0.25), \,
 d_6 = (0.25,-0.25), \, d_7 = (-0.25,0.25), \]
\[ d_8 = (-0.25,-0.25),\, d_9 = (0.75,0.6875), \, d_{10} = (-0.75,-0.75) \]
with weights
\[ \alpha_1 = 0.24497863, \, \alpha_2 = 0.59721306, \, \alpha_3 = 0.69141129, \, \alpha_4 = 0.23000000, \]
\[ \alpha_5 = 0.22500000, \,  
\alpha_6 = 0.22500000, \, \alpha_7 = 0.04500000, \]
\[ \alpha_8 = 0.04500000, \, \alpha_9 = 0.36462036, \, \alpha_{10} =  0.47337968. \]
At the Dirac masses, the potential should take on the values
\[ u_{ex}(d_1) = u_{ex}(d_2) = u_{ex}(d_3) = u_{ex}(d_4)= 1, \, u_{ex}(d_5) = u_{ex}(d_6) = 0.85, \]
\[ u_{ex}(d_7) = u_{ex}(d_8) = 0.9, \, u_{ex}(d_9) = u_{ex}(d_{10}) = 1.2. \]

In Figure~\ref{fig:ThreeDirac}, we picture the computed potential functions $u$, the images of each Dirac mass (with the colour indicating the relative weight assigned to each mass), and the error in the computed potential at the masses.  The error is also presented in Table~\ref{table:ThreeDirac}.  In each case, we observe convergence on the order of approximately $\sqrt{h}$.

Table~\ref{table:ThreeDirac} also provides the number of Newton iterations required for convergence, which depends weakly (approximately $\bO(M^{0.3})$ on average) on the total number of discretisation points $M = N_X^2$.

\begin{figure}[htdp]
	\centering
	\subfigure[]{\includegraphics[width=.3\textwidth]{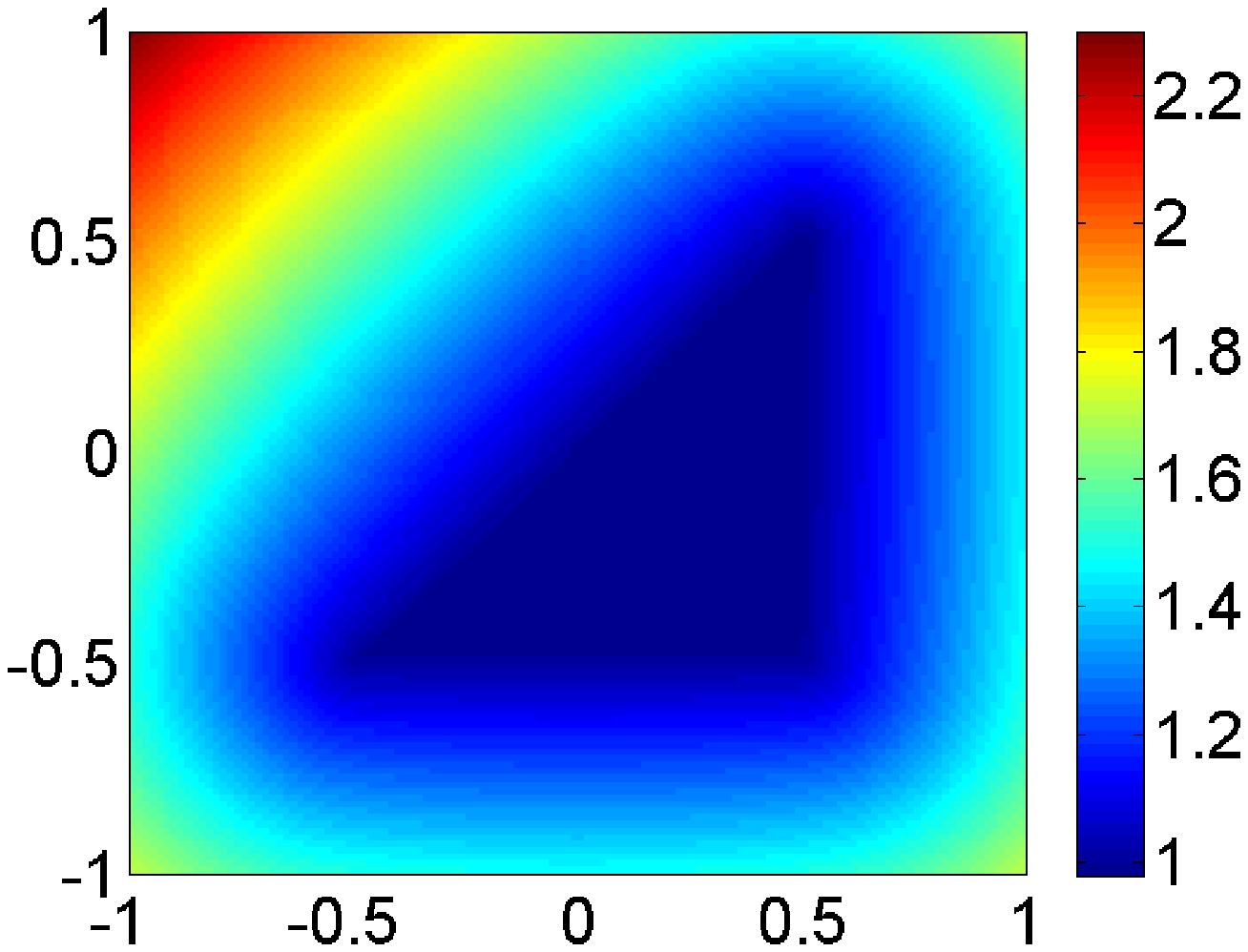}\label{fig:ThreeDiracSol}}
	\subfigure[]{\includegraphics[width=.3\textwidth]{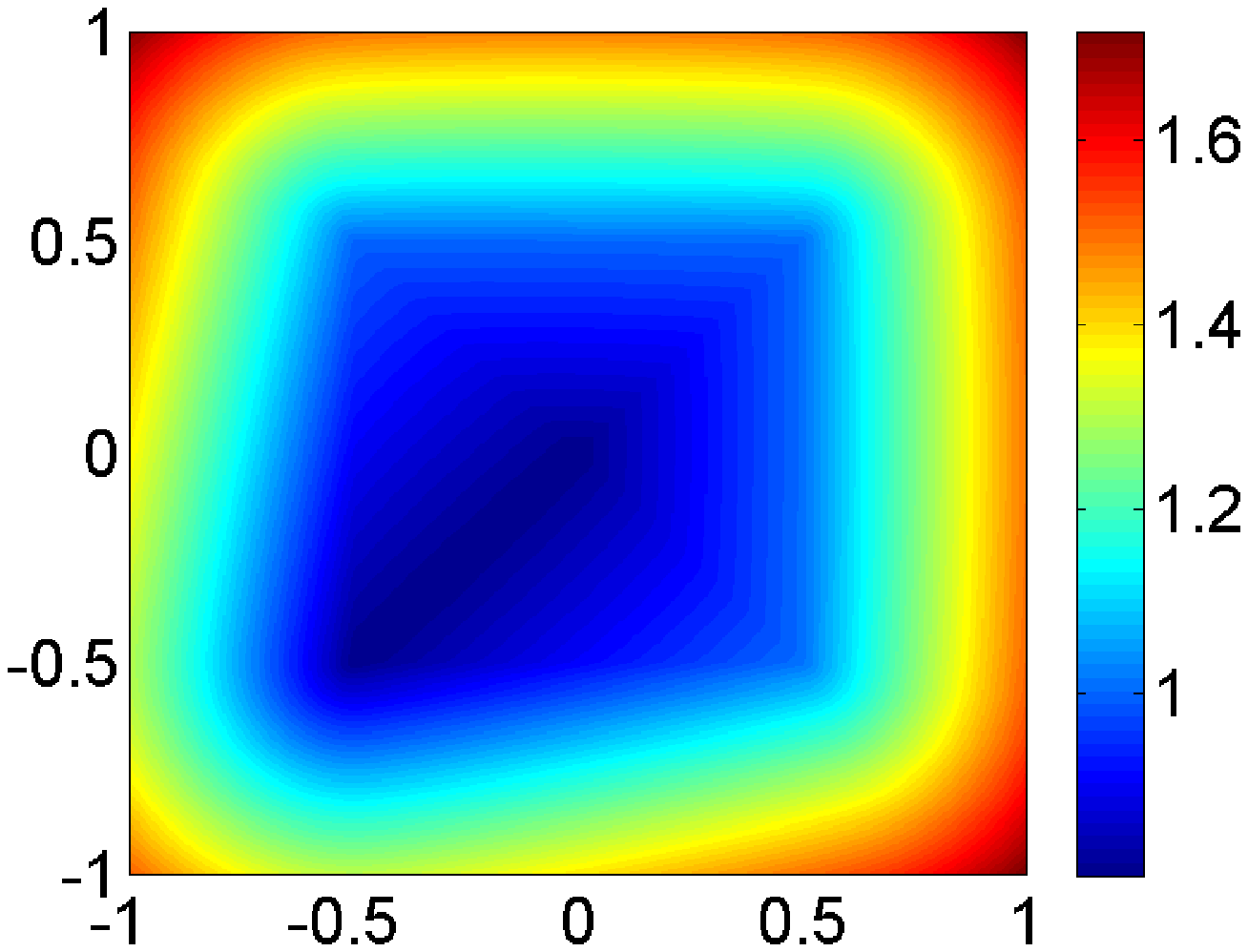}\label{fig:FiveDiracSol}}
	\subfigure[]{\includegraphics[width=.3\textwidth]{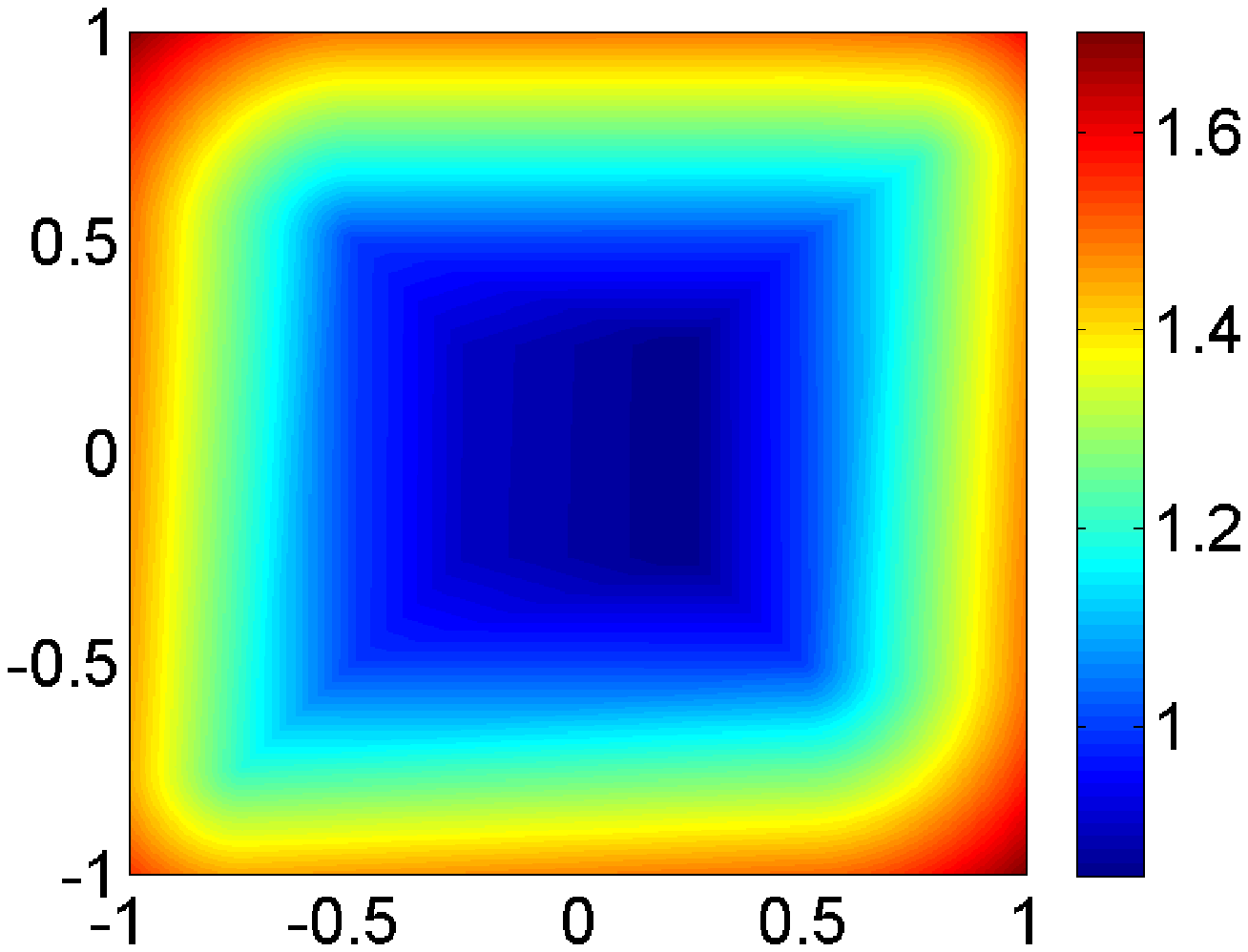}\label{fig:TenDiracSol}}
  \subfigure[]{\includegraphics[width=.3\textwidth]{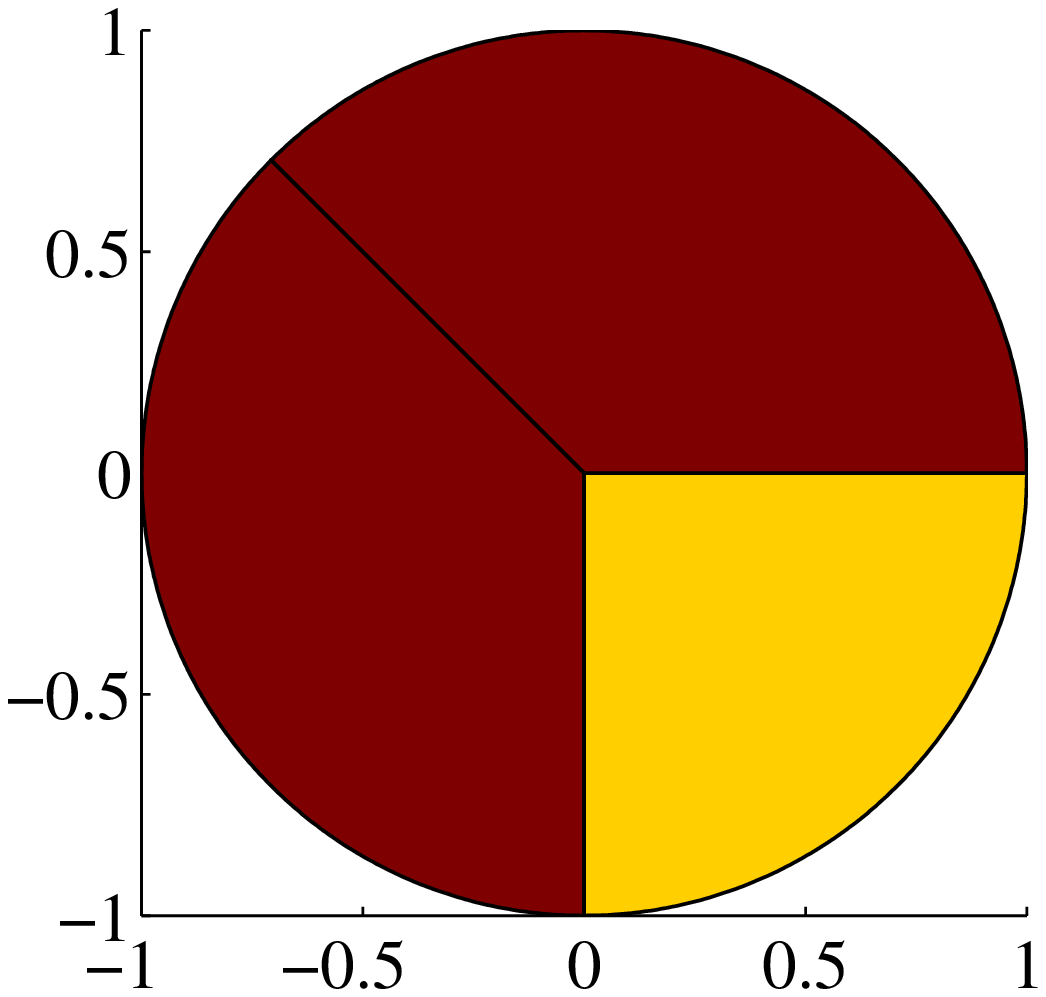}\label{fig:ThreeDiracRegions}}
  \subfigure[]{\includegraphics[width=.3\textwidth]{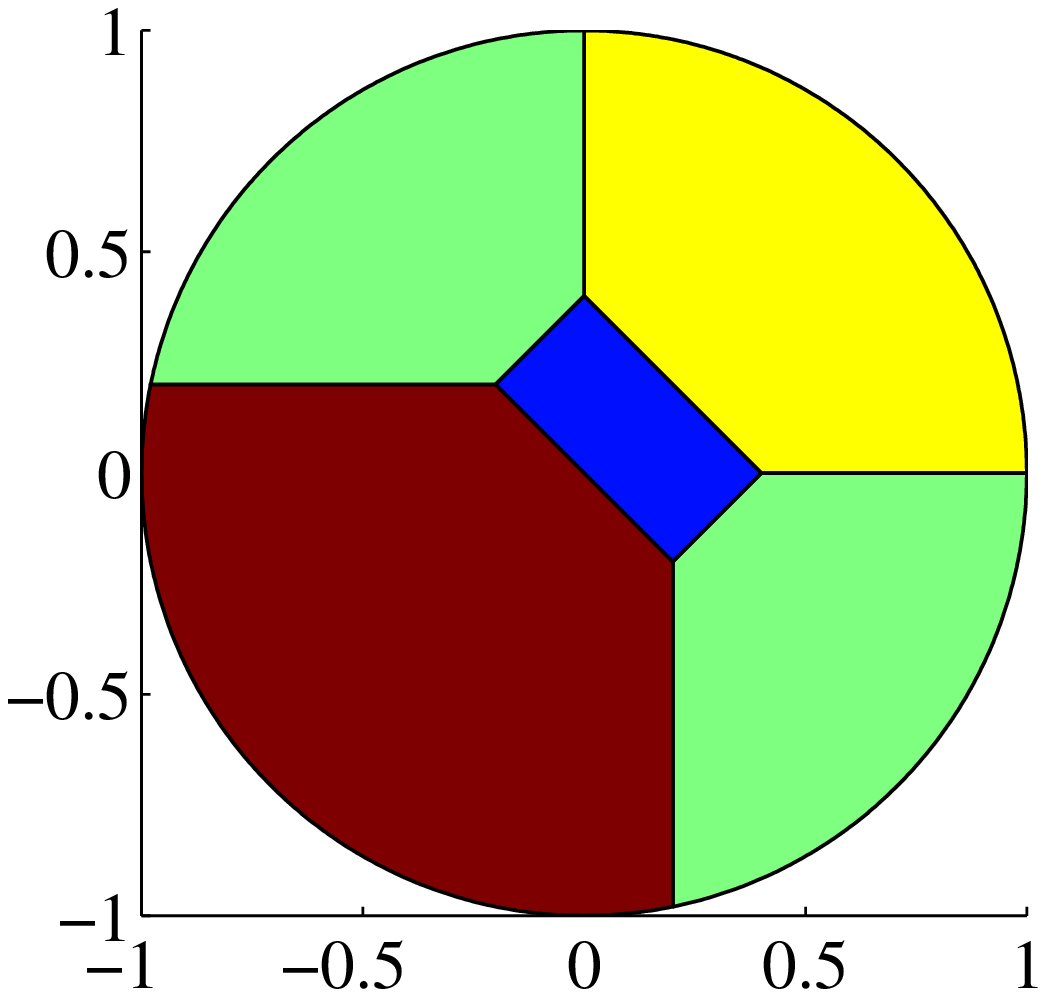}\label{fig:FiveDiracRegions}}
	\subfigure[]{\includegraphics[width=.3\textwidth]{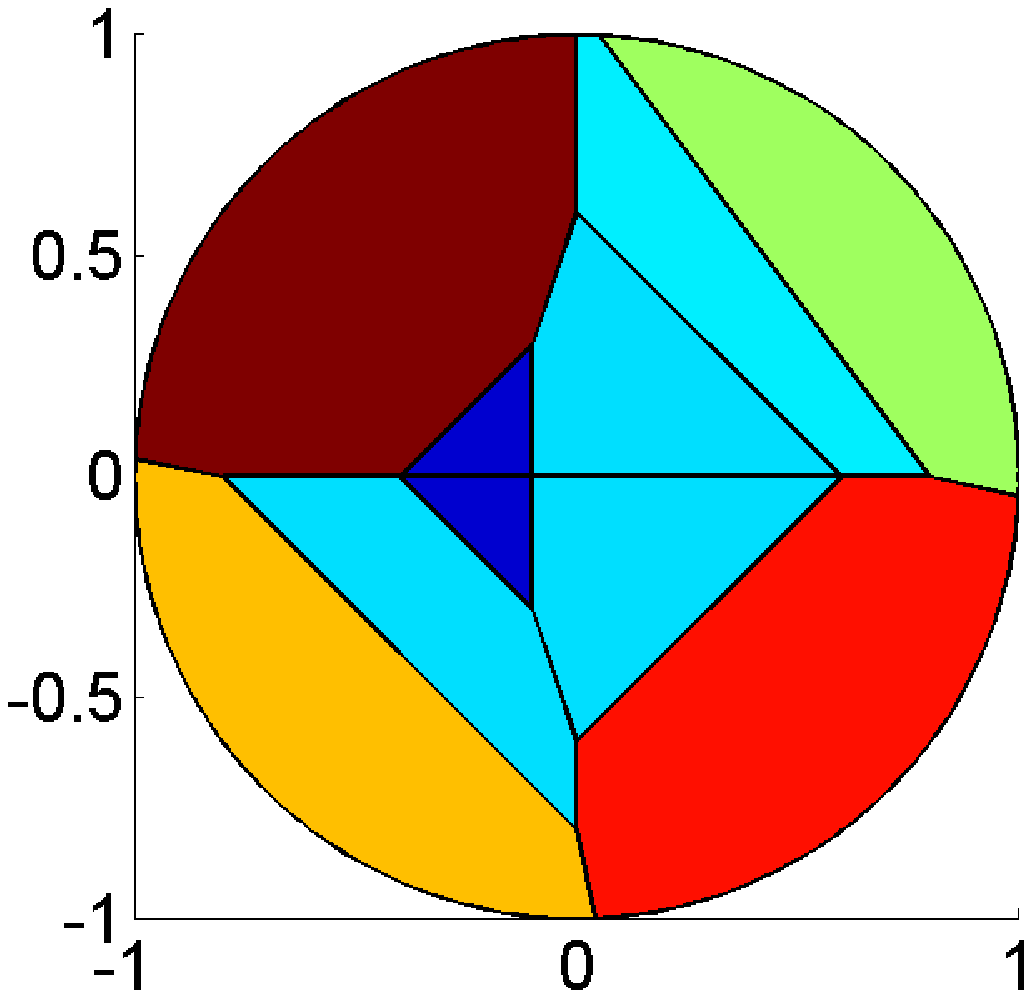}\label{fig:TenDiracRegions}}
  \subfigure[]{\includegraphics[width=.45\textwidth]{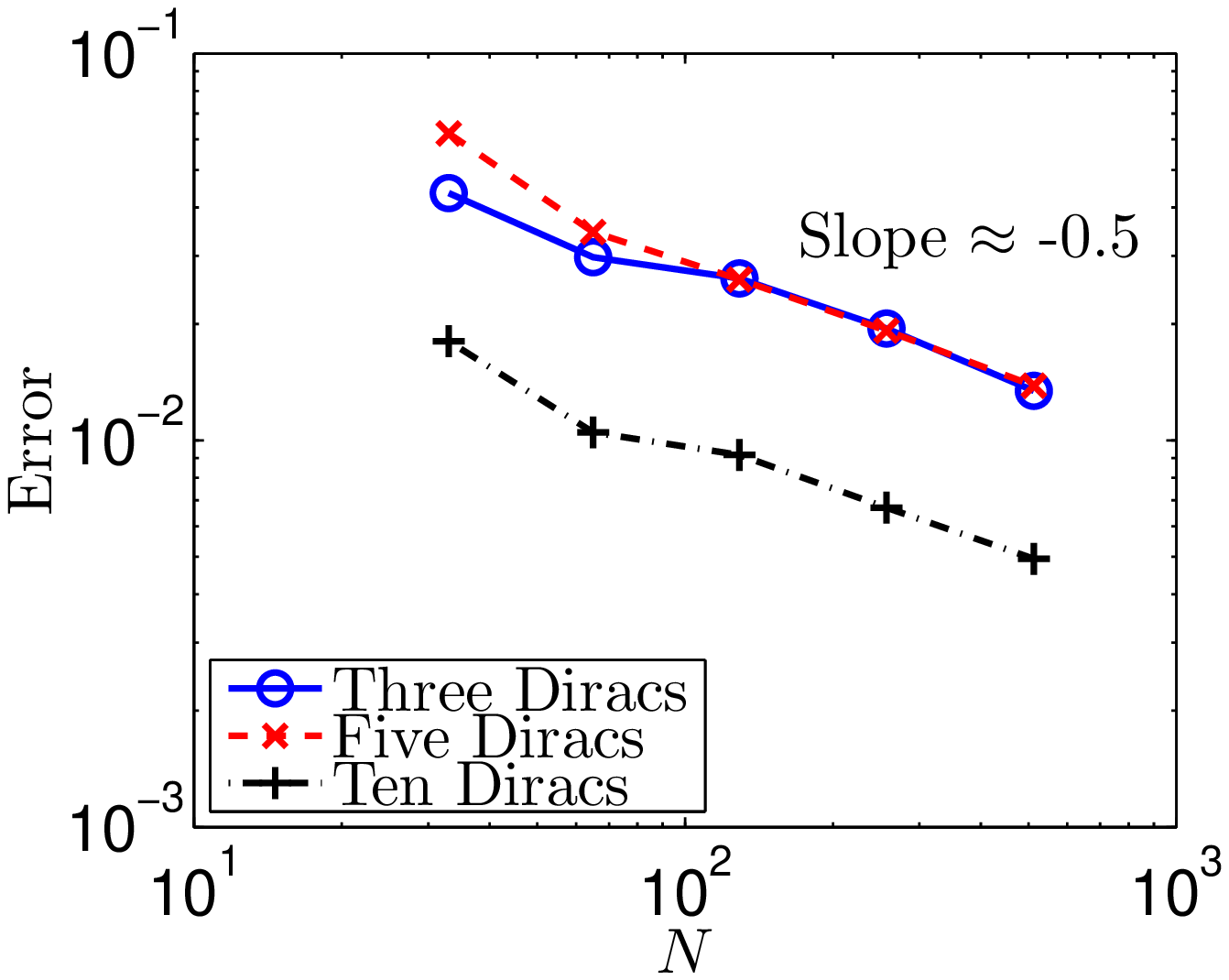}\label{fig:DiracHeightError}}
  	\caption{\subref{fig:ThreeDiracSol},\subref{fig:FiveDiracSol},\subref{fig:TenDiracSol}~Potential functions and \subref{fig:ThreeDiracRegions},\subref{fig:FiveDiracRegions},\subref{fig:TenDiracRegions}~images of Dirac masses for three, five, and ten Dirac masses respectively. \subref{fig:DiracHeightError}~Maximum error in potential at the Diracs.}
  	\label{fig:ThreeDirac}
\end{figure}

\begin{table}[htdp]\small
\begin{center}
\begin{tabular}{c|ccc|ccc}
$N_X$ &\multicolumn{3}{c}{Maximum Error} & \multicolumn{3}{c}{Newton Iterations} \\
  & {Three Diracs} & {Five Diracs} & Ten Diracs & {Three Diracs} & {Five Diracs} & Ten Diracs\\
\hline
33  & $4.36\times10^{-2}$ & $6.21\times10^{-2}$ & $1.80\times10^{-2}$ & 3 & 11 & 10\\
65  & $2.98\times10^{-2}$ & $3.45\times10^{-2}$ & $1.05\times10^{-2}$ & 10 & 9 & 20\\
129 & $2.62\times10^{-2}$ & $2.61\times10^{-2}$ & $0.92\times10^{-2}$ & 13 & 17 & 26\\
257 & $1.95\times10^{-2}$ & $1.92\times10^{-2}$ & $0.67\times10^{-2}$ & 18 & 23 & 31\\
513 & $1.34\times10^{-2}$ & $1.38\times10^{-2}$ & $0.49\times10^{-2}$ & 44 & 25 & 36\\
\end{tabular}
\end{center}
\caption{Maximum error in the computed solution at the Dirac masses and number of Newton iterations for three, five, and ten Diracs.}
\label{table:ThreeDirac}
\end{table}

\subsection{Multiple Diracs}\label{sec:resultsMultiple}
Finally, we randomly position multiple Diracs in the domain and map these onto the unit circle.  The computed solutions are displayed in Figure~\ref{fig:MultiDirac}.  To demonstrate the superiority of the mixed Aleksandrov-viscosity solution, we also provide results computed with the traditional viscosity solver.  Even qualitatively, it is clear that the viscosity solver leads to large errors in the computed cell areas.  The errors in the cell areas are presented in Table~\ref{table:MultiDirac}.  Again, we observe convergence on the order of approximately $\sqrt{h}$ for the Aleksandrov scheme.

\begin{figure}[htdp]
	\centering
	\subfigure[]{\includegraphics[width=.32\textwidth]{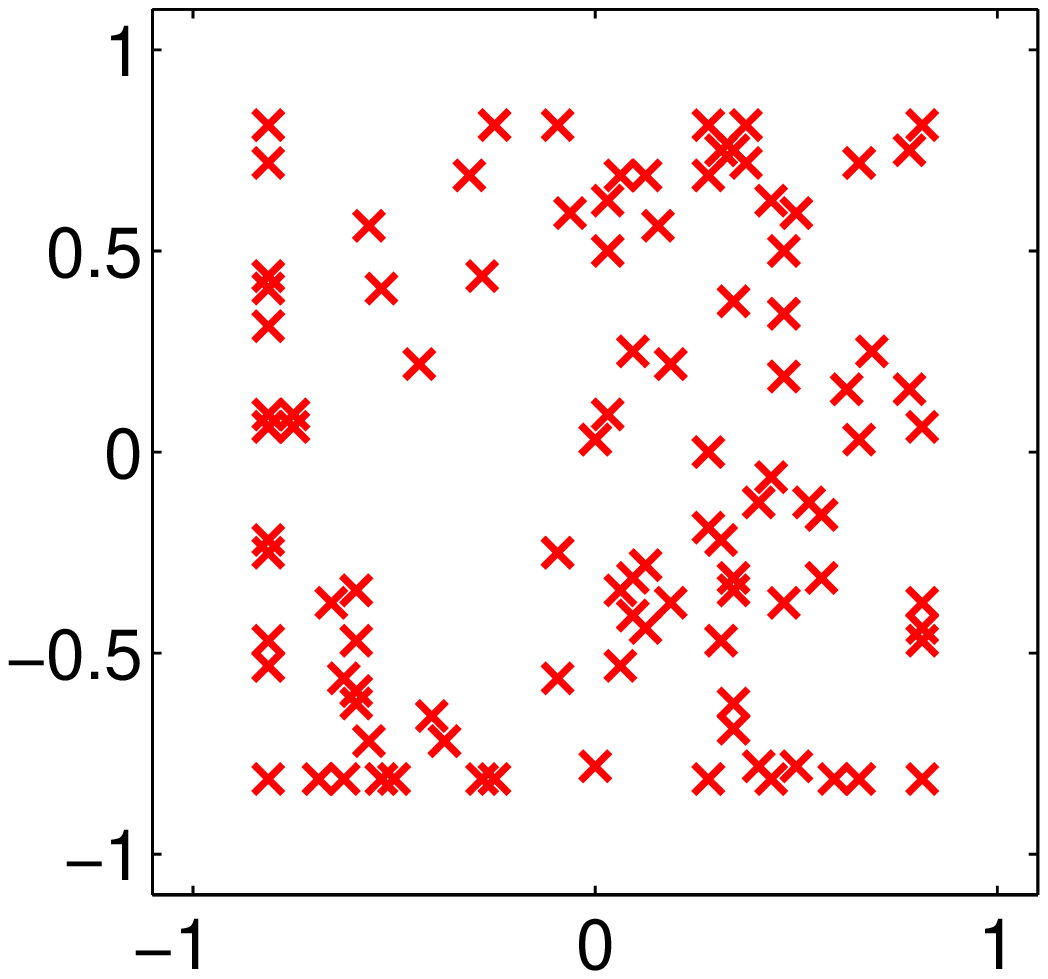}\label{fig:dirac100loc}}
	\subfigure[]{\includegraphics[width=.32\textwidth]{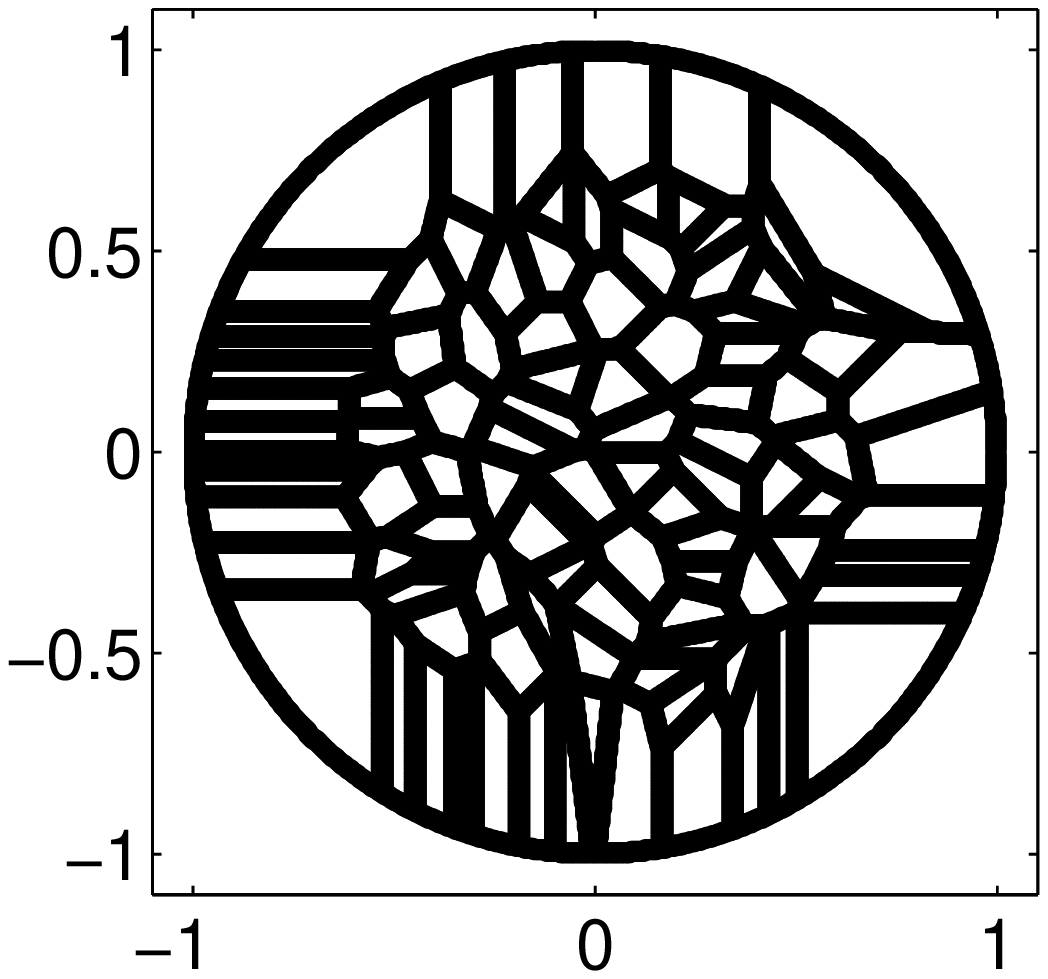}\label{fig:dirac100Visc}}
  \subfigure[]{\includegraphics[width=.32\textwidth]{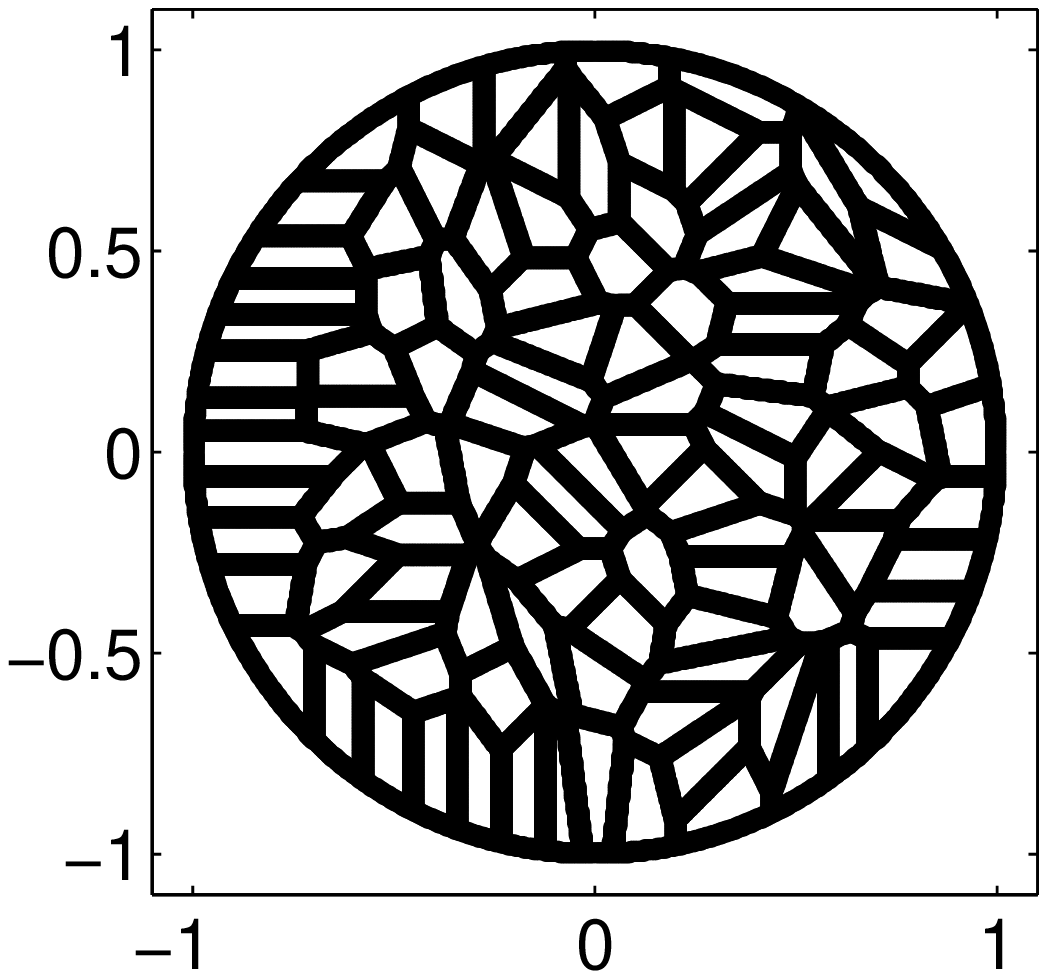}\label{fig:dirac100regions}}
		\subfigure[]{\includegraphics[width=.4\textwidth]{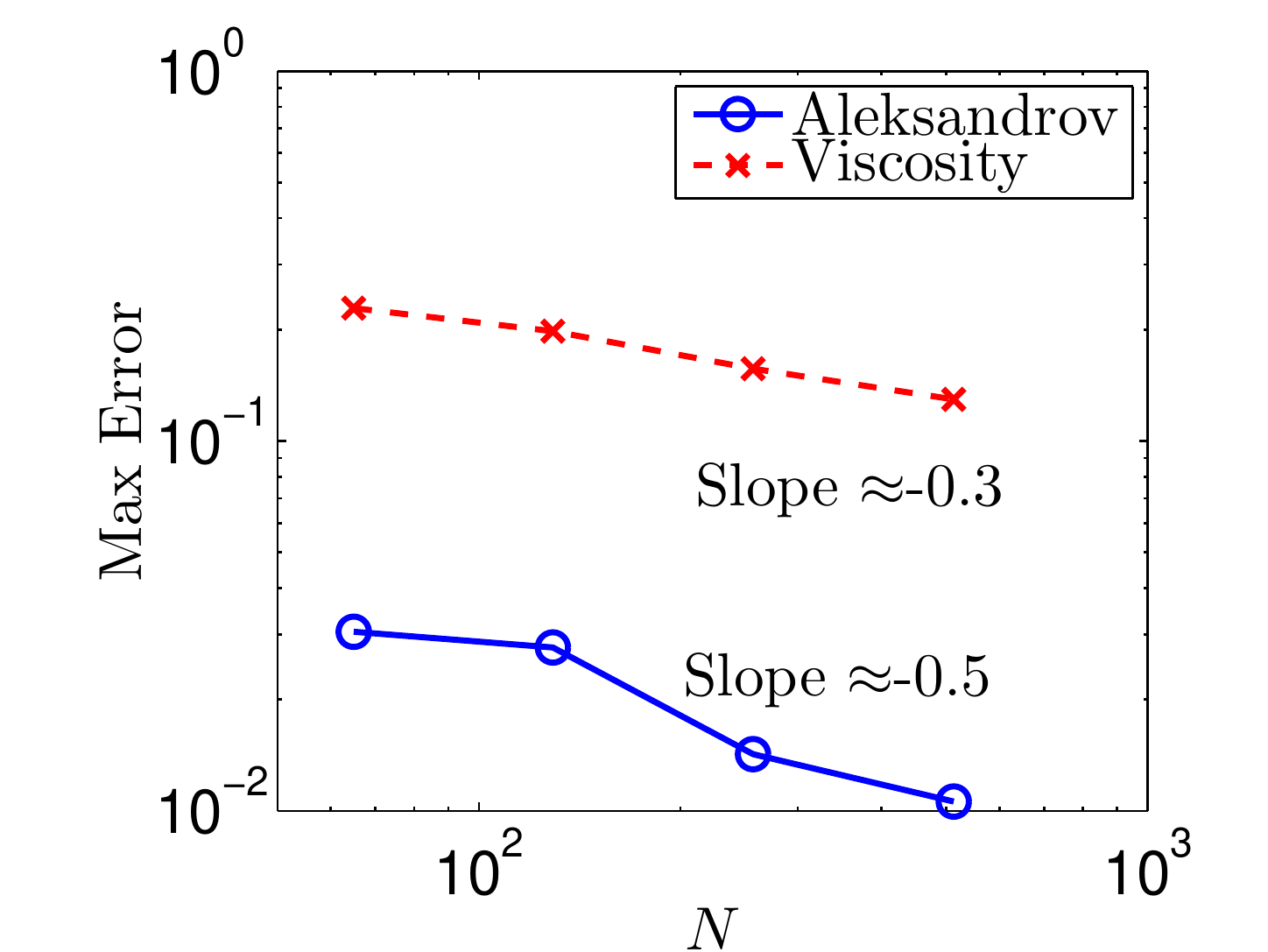}\label{fig:dir100ErrInf}}
	\subfigure[]{\includegraphics[width=.4\textwidth]{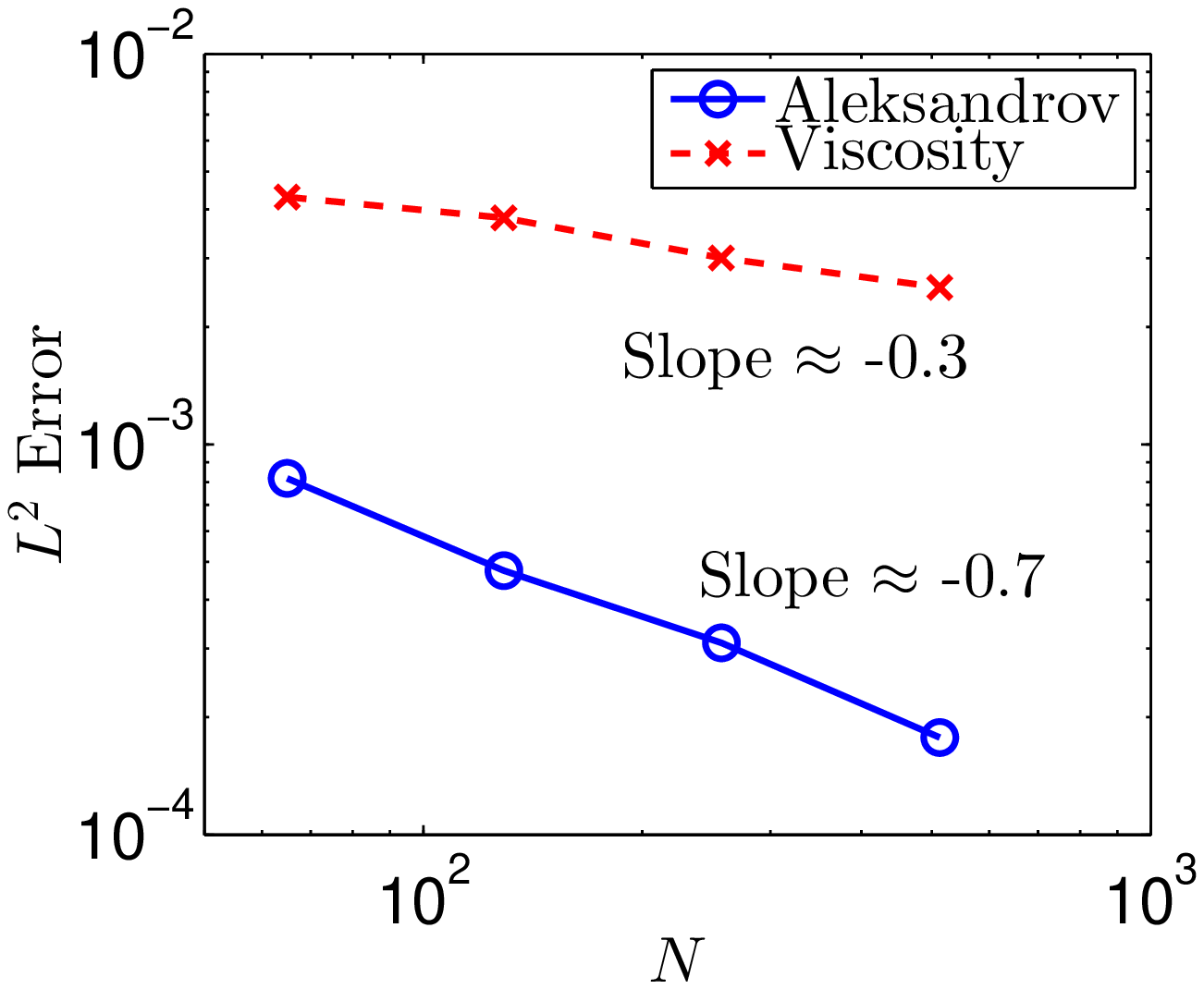}\label{fig:dir100Err2}}
  	\caption{\subref{fig:dirac100loc}~One hundred randomly positioned Dirac masses.  Computed images of these masses using \subref{fig:dirac100Visc}~a traditional viscosity scheme and \subref{fig:dirac100regions}~the mixed Aleksandrov-viscosity scheme.  Errors in computed cell areas measured in \subref{fig:dir100ErrInf}~$L^\infty$ and \subref{fig:dir100Err2}~$L^2$.}
  	\label{fig:MultiDirac}
\end{figure} 

\begin{table}[htdp]\small
\begin{center}
\begin{tabular}{c|cc|cc}
$N_X$ &\multicolumn{2}{c}{Maximum Error} &\multicolumn{2}{c}{$L^2$ Error} \\
  & Aleksandrov-viscosity & Viscosity & Aleksandrov-viscosity & Viscosity \\
\hline
65  & 0.030 & 0.228 & $0.82\times10^{-3}$ & $4.30\times10^{-3}$\\
129 & 0.028 & 0.198 & $0.47\times10^{-3}$ & $3.81\times10^{-3}$ \\
257 & 0.014 & 0.157 & $0.31\times10^{-3}$ & $3.01\times10^{-3}$\\
513 & 0.011 & 0.130 & $0.18\times10^{-3}$ & $2.52\times10^{-3}$ 
\end{tabular}
\end{center}
\caption{Error in the areas of the computed cells for one hundred randomly positioned Dirac masses.}
\label{table:MultiDirac}
\end{table}

\section{Conclusions} 
Existing techniques for computing Pogorelov solutions of the \MA equation rely on geometric methods that have cubic complexity in the number of Dirac masses.  The problem can also be expressed as a well-behaved convex optimisation problem if a good initialisation is available~\cite{dot}.

We introduced a new mixed Aleksandrov-viscosity formulation of this problem and provided a local characterisation of the equation.  Using this new formulation, we constructed a monotone finite difference approximation that can be solved using Newton's method.  Experimentally, the number of Newton iterations depended weakly (approximately~$\bO(M^{0.3})$) on the total number of discretisation points and the accuracy of the computed solution was on the order of $\sqrt{h}$.

A possible application of this method would be 
to provide a good initialisation for an exact geometric algorithm.  The extension to non-constant density functions 
also appears straightforward (Appendix~\ref{app:extension}).  The proof of a comparison principle for this mixed Alexandrov-viscosity formulation
remains open.

{
\appendix

\section{Convextiy}\label{app:convexity}
To assist in the approximation of the \MA equation, we can absorb the convexity constraint into the operator as proposed in~\cite{FroeseTransport}.

We denote by
\bq\label{eq:posneg}
u^+ = \max\{u,0\}, \quad u^- = \max\{-u,0\}
\eq 
the positive and negative parts of the function $u$.  Then the convexified \MA operator can be defined as
\begin{multline}\label{eq:MA_convex}
F(x,u(\cdot),\nabla u(x),D^2u(x)) = \\ \begin{cases}
-\min\limits_{\abs{\nu} = 1, \nu\cdot\nu^\perp = 0}\left\{u_{\nu\nu}^+u_{\nu^\perp\nu^\perp}^+-u_{\nu\nu}^- - u_{\nu^\perp\nu^\perp}^-\right\}, & x\in X \backslash\bigcup\limits_{k=1}^K\{d_k\}\\
-M[u](d_k) + \alpha_k, & k=1,\ldots,K\\
H(\nabla u(x)) - \langle u \rangle, & x\in\partial X .
\end{cases}
\end{multline}

The Aleksandrov-viscosity solution notion is easily adapted to this new operator.
\begin{definition}[Mixed Aleksandrov-viscosity solution]
A Lipschitz continuous function $u$ is a \emph{subsolution (supersolution)} of~\eqref{eq:MA_convex} if 
\begin{enumerate}
\item For every $x_0 \notin \bigcup\limits_k\{d_k\}$ and smooth function $\phi$, if $u-\phi$ has a local maximum (minimum) at $x_0$ then 
\[ 
F(x_0,u(\cdot),\nabla\phi(x_0),D^2\phi(x_0)) \leq(\geq)  0
\]
\item $-M[u](d_k) +\alpha_k\geq (\leq) 0$.
\end{enumerate}
A function is a \emph{mixed Aleksandrov-viscosity solution} if it is both a subsolution and a supersolution.
\end{definition}

\begin{theorem}[Equivalence of solutions]\label{thm:ViscAleksConv}
A function $u$ with mean zero is a solution of~\eqref{eq:MA_convex} if and only if it is a convex Aleksandrov solution of~\eqref{eq:MA_Aleks}.
\end{theorem}

\begin{proof}
The proof follows immediately from Theorem~\ref{thm:ViscAleks}, Lemma~\ref{lem:ViscIsConv}, and Lemma~\ref{lem:ConvIsVisc}.
\end{proof}

%

\begin{lemma}\label{lem:ViscIsConv}
Let $u$ be a convex solution of~\eqref{eq:MA_visc}.  Then $u$ is a mixed Aleksandrov-viscosity solution of~\eqref{eq:MA_convex}.
\end{lemma}
\begin{proof}
First we show that $u$ is a subsolution of the convexified \MA equation.  Let $x_0\notin\bigcup\limits_k\{d_k\}$, $\phi\in C^2$ , and suppose $u-\phi$ has a local maximum at $x_0$.  Since $u$ is a convex function, the test function $\phi$ must be also be convex in a neighbourhood of $x_0$.  Then we have
\[ 
F(x_0,u(\cdot),\nabla\phi(x_0),D^2\phi(x_0)) \leq  0
\]
which has the correct sign since $u$ is a viscosity subsolution of the original \MA equation~\eqref{eq:MA_visc}.

Additionally, $-M[u](d_k)+\alpha_k = 0$ since $u$ is a solution of~\eqref{eq:MA_visc}.

Next we verify that $u$ is a supersolution.  Let $\phi\in C^2$ and suppose that $u-\phi$ has a local minimum at a point $x_0$.  If $x_0\in\partial X$ then we require $\nabla u(x_0)\in\partial u(x_0)$ in order to achieve a local minimum.  Thus $H(\nabla\phi(x_0)) = 0$ since $\partial u(\partial X) \subset \partial B(0,1)$ (See Lemma~\ref{lem:ViscAleks}).

Now we check the condition for    $x_0\in X\backslash\bigcup\limits_k\{d_k\}$.
If $\phi$ is convex near $x_0$, we can repeat the argument used for the subsolution property to verify that the convexified \MA operator has the correct sign.  Otherwise, $\phi$ is non-convex.
Then there exists a direction $\nu$ such that $\phi_{\nu\nu}<0$ in a neighbourhood of $x_0$.  Thus 
\[  -\min\limits_{\abs{\nu} = 1, \nu\cdot\nu^\perp = 0}\left\{\phi_{\nu\nu}^+\phi_{\nu^\perp\nu^\perp}^+-\phi_{\nu\nu}^- - \phi_{\nu^\perp\nu^\perp}^-\right\} > 0\]
near $x_0$ so that
\[F(x_0,u(\cdot),\nabla\phi(x_0),D^2\phi(x_0)) \geq  0\]
which yields the desired inequality.  

As before, $M[u]$ takes on the correct values at the diracs.
\end{proof}

\begin{lemma}\label{lem:ConvIsVisc}
Let $u$ be a viscosity solution of~\eqref{eq:MA_convex}.  Then $u$ is a convex viscosity solution of~\eqref{eq:MA_visc}.
\end{lemma}

\begin{proof}
Choose any smooth $\phi$ such that $u-\phi$ has a local maximum at some point $x_0\in X\backslash\bigcup\limits_k \{d_k\}$.  Since $u$ is a viscosity solution of~\eqref{eq:MA_convex}, we require
\[ F(x_0,u(\cdot),\nabla\phi(x_0),D^2\phi(x_0)) =  -\min\limits_{\abs{\nu} = 1, \nu\cdot\nu^\perp = 0}\left\{\phi_{\nu\nu}^+\phi_{\nu^\perp\nu^\perp}^+-\phi_{\nu\nu}^- - \phi_{\nu^\perp\nu^\perp}^-\right\} \leq 0.\]
As demonstrated in the proof of Lemma~\ref{lem:ViscIsConv}, this is only possible if $\phi$ is convex.  Since only convex test functions allow for local maxima, we conclude that $u$ is also a convex function.

We need only check the definition at points $x_0\in X\backslash\bigcup\limits_k \{d_k\}$ since the operator is unchanged in the remainder of the domain.  We show that $u$ is a subsolution of~\eqref{eq:MA_visc}; the supersolution property is similar.  Now we choose any convex $\phi\in C^2$ such that $u-\phi$ has a maximum at $x_0$.  Since $u$ is a viscosity solution of~\eqref{eq:MA_convex}, we know that
\begin{align*}
F(x_0,u(\cdot),\nabla\phi(x_0),D^2\phi(x_0)) &=  -\min\limits_{\abs{\nu} = 1, \nu\cdot\nu^\perp = 0}\left\{\phi_{\nu\nu}^+\phi_{\nu^\perp\nu^\perp}^+-\phi_{\nu\nu}^- - \phi_{\nu^\perp\nu^\perp}^-\right\}\\
  &=  -\min\limits_{\abs{\nu} = 1, \nu\cdot\nu^\perp = 0}\left\{\phi_{\nu\nu}^+\phi_{\nu^\perp\nu^\perp}^+-\phi_{\nu\nu}^-\right\}\\
	&\leq 0,
\end{align*}
where the two operators are equivalent since $\phi$ is convex.  
\end{proof}

\section{Discretisation of Monge-Amp\`ere} \label{app:MA}

We briefly review the approximation of the \MA operator and boundary conditions, which have been fully described in previous works~\cite{BFO_OTNum,FroeseTransport}.
\bq\label{eq:MA_BC}
\begin{cases}
 -{\det}^+(D^2u(x)) \equiv  -\min\limits_{\abs{\nu}=1}\left\{(u_{\nu\nu})^+(u_{\nu^\perp\nu^\perp})^+-(u_{\nu\nu})^- - (u_{\nu^\perp\nu^\perp})^-\right\},& \quad x\in X \\
H(\nabla u), & x\notin X .
\end{cases}
\eq
We let $h$ denote the spatial resolution of the grid and introduce an angular discretisation $\theta_i, \, i=1,\ldots,N$ of the periodic interval $[0,2\pi)$. This will not be uniform, but is instead chosen so that whenever the point $x$ lives on the grid, there exists a value $l_i>0$ such that $x\pm l_ihs_{\theta_i}$ also lies on the grid. 

In the discretised problem, the minimum in the \MA operator~\eqref{eq:MA_BC} is computed over the directions $\nu_i = s_{\theta_i}$ and the needed second directional derivatives are approximated by
\bq\label{eq:D2_disc}
u_{\nu_i\nu_i}(x) \approx \frac{u(x+l_i hs_{\theta_i}) + u(x-l_ihs_{\theta_i})-2u(x)}{l_i^2h^2},
\eq
which leads to a consistent, monotone approximation $(MA)^h[u]$ of the \MA operator.
 
To include the mean-zero constraint, we use the approximation
\bq\label{eq:discMA}
(MA)^h[u_{i,j}] - h^2\sum\limits_{k,l} u_{k,l} + h^2u_{i,j} = 0.
\eq 
Note that this approximation is consistent and preserves monotonicity.

To avoid the need to compute the mean at each step, we note that this is equivalent to solving the discrete equation
\bq\label{eq:discMA2} (MA)^h[v_{i,j}] + h^2v_{i,j} = 0 \eq
and setting
\bq\label{eq:scaleMA} u_{i,j} = v_{i,j} + \frac{h^2}{h^2-1}\sum\limits_{k,l}v_{k,l}. \eq

To approximate the Hamilton-Jacobi equation, we express the signed distance function in terms of the supporting hyperplanes to the target set $Y$.
\bq\label{HJoblique}
H(\grad u(x)) = \sup\limits_{\abs{n}=1}\{ \grad u(x) \cdot n - H^*(n) \mid n\cdot n_x > 0\}
\eq
where the Legendre-Fenchel transform is explicitly given by
\bq\label{LF}
H^*(n) = \sup\limits_{y_0\in\partial Y}\{y_0\cdot n\}
\eq
and $n_x$ is the unit outward normal to the domain at a boundary point $x\in X$.

The Legendre-Fenchel transforms are pre-computed using a finite set of $N_Y$ directions
\[ n_j = \left(\cos\left({2\pi j/N_Y}\right),\sin\left({2\pi j/N_Y}\right)\right), \quad j = 1,\ldots,N_Y. \]
Then the advection terms in the Hamilton-Jacobi equation are approximated by a simple upwinding scheme 
\begin{multline} \nabla u(x)\cdot n \approx \max\{n_1,0\}\frac{u(x) - u(x-he_1 )}{h} + \min\{n_1,0\}\frac{u(x+he_1)-u(x)}{h} \\
+ \max\{n_2,0\}\frac{u(x) - u(x-he_2 )}{h} + \min\{n_2,0\}\frac{u(x+he_2)-u(x)}{h}. 
\end{multline}
At boundary points, values outside the domain will not contribute to the supremum in~\eqref{HJoblique} and can be set to zero.

In addition to boundary points, we also set this Hamilton-Jacobi equation in a layer around the set $X$ that is just large enough to accommodate the wide stencil discretisation of the \MA operator.

\section{Extension to non-constant densities} \label{app:extension}

The approach described in this paper can be extended to more general measures.  We can consider a target measure
\[d\nu(y) = g(y) \, dy\] 
where $g$ is a positive, Lipschitz continuous density supported on a convex set $Y$.  We can also allow the source measure to have a non-constant background density $f(x)$
\[ \mu(E) = \sum\limits_{k=1}^K\delta_{d_k}(E) + \int_E f(x)\,dx. \]

The Alexandrov formulation of the Monge-Amp\`ere equation (Definition~\ref{def:aleks}) is now
\[ A_{u}(E) = \mu(E) \] 
for every measurable $E\in\R^2$.
Here
\[
A_{u} (E)  = \int_{\partial u(E)} g(y) \, dy \]

To develop the local characterisation of this measure at the dirac points, we consider
\[
M[u](d_k) = \int_{\partial u(d_k)} g(y) \, dy .
\]
As in Theorem~\ref{thm:subgradient}, we can convert the subgradient into polar coordinates to obtain an expression of the form
\[
M[u](d_k) =  \int_{0}^{2\pi} \int_{\min(R_-(x,\theta'),R_+(x,\theta'))}^{ R_+(x,\theta') } g(r,\theta)\, r \, dr \, d\theta'.
\]

Now for each $\theta$ we let $G_{\theta}(r)$ be an antiderivative of \[r \mapsto g (r,\theta) \, r.\]  Note that for all $\theta$, $G_{\theta}$ is a non-decreasing function
when $r\in \R^+$ and $g \ge 0 $. 
We therefore obtain the local characterisation
\[
M[u](d_k) =  \int_{0}^{2\pi} \left[G_\theta(R_+(x,\theta')) -  G_\theta(R_-(x,\theta'))  \right]^+ \, d\theta
\]
where $R_-, R_+$ are defined as in~\eqref{eq:Bp}-\eqref{eq:Bm}.

The boundary condition will be defined in terms of $H(y)$, the signed-distance to the boundary of the convex target $Y$.

Thus the \MA equation we need to solve is
\bq
F(x,u(\cdot),\nabla u(x), D^2 u(x)) = 
\begin{cases}
-\det(D^2u(x)) + f(x)/g(\nabla u(x)) , & x\in X \backslash\bigcup\limits_{k=1}^K\{d_k\}\\
-M[u](d_k) + \alpha_k, & k=1,\ldots,K\\
H(\nabla u(x)) - \langle u \rangle , & x\in\partial X .
\end{cases}
\eq
A mixed Aleksandrov-viscosity solution is defined for this equation as in Definition~\ref{def:mixed}.

As in~\autoref{sec:discMeasure}, finite difference approximations of $\partial_\theta  u (d_k)$ lead to a consistent, monotone discretisation 
of $M[u](d_k)$.  Note that the implementation will now require the computation of the functions $G_{\theta}$.

}

\bibliographystyle{amsplain}
\bibliography{AleksVisc}

\end{document}